\documentclass[11pt]{amsart}
\usepackage{hyperref,amssymb,tikz,amsfonts,amssymb}
\usepackage[margin=2.5cm]{geometry}
\usetikzlibrary{arrows,shapes,positioning,decorations.markings, decorations.pathreplacing,decorations.pathmorphing}
\newcommand{\Cay}{\mathop{\mathrm{Cay}}}
\newcommand{\Aut}{\mathop{\mathrm{Aut}}}
\def\Zent#1{{\bf Z}({{#1}})}
\def\Z#1{{\bf Z}({{#1}})}
\def\cent#1#2{{\bf C}_{{#1}}({{#2}})}
\def\centm#1#2{{\bf C}_{{#1}}^{-}({{#2}})}

\def\nor#1#2{{\bf N}_{{#1}}({{#2}})}

\numberwithin{equation}{section}
\newtheorem{theorem}{Theorem}[section]
\newtheorem{lemma}[theorem]{Lemma}
\newtheorem{corollary}[theorem]{Corollary}

\newtheorem{conjecture}[theorem]{Conjecture}

\newtheorem*{facts}{Facts}
\theoremstyle{definition}

\newtheorem{notation}[theorem]{Notation}
\newtheorem{remark}[theorem]{Remark}

\begin{document}
\title[Bipartite Cayley graphs]{A conjecture on bipartite graphical regular representations}

\author{Jia-Li Du}
\address{Jia-Li Du, Department of Mathematics, China University of Mining and Technology, Xuzhou 221116, China}
\email{dujl@cumt.edu.cn}

\author{Yan-Quan Feng}
\address{Yan-Quan Feng, Department of Mathematics, Beijing Jiaotong University, Beijing 100044, China}
\email{yqfeng@bjtu.edu.cn}

\author{Pablo Spiga}
\address{Pablo Spiga, Dipartimento di Matematica e Applicazioni, University of Milano-Bicocca, Via Cozzi 55, 20125 Milano, Italy}
\email{pablo.spiga@unimib.it}
\begin{abstract}
In this paper we are concerned with the classification of the finite groups admitting a bipartite DRR and a bipartite GRR.

First, we find a natural obstruction in a finite group for not admitting a bipartite GRR. Then we give a complete classification of the finite groups satisfying this natural obstruction and hence not admitting a bipartite GRR.
Based on these results and on some extensive computer
computations, we state a conjecture aiming to give a complete classification of
the finite groups admitting a bipartite GRR.

Next, we prove the existence of bipartite DRRs for most of the finite groups not admitting a bipartite GRR found in this paper. Actually, we prove a much stronger result: we give an asymptotic enumeration of the bipartite DRRs over these groups. Again, based on these results and on some extensive computer
computations, we state a conjecture aiming to give a complete classification of
the finite groups admitting a bipartite DRR.

\smallskip

\end{abstract}

\maketitle
\section{Introduction}
Let $R$ be a group and let $S$ be a subset of $R$. The \textit{Cayley digraph} on $R$ with
connection set $S$, denoted by $\Cay(R, S)$
in this paper, is the digraph with vertex-set $R$ and with
$(g, h)$ being an arc if and only if $gh^{-1} \in S$. Actually, $\Cay(R, S)$ is a graph if and only if $S$ is inverse-closed (that is, $S^{-1}:=\{s^{-1}\mid s\in S\}=S$), in which
case it is called a \textit{Cayley graph}. It was already observed by Cayley that the group $R$ acts regularly as a
group of automorphisms on $\Cay(R, S)$ by right multiplication. Therefore, we may identify $R$ as a subgroup of the automorphism group $\Aut(\Cay(R,S))$ of $\Cay(R,S)$.

When $R$ equals $\Aut(\Cay(R,S))$, $\Cay(R, S)$
is called a \textit{DRR} (for \textit{digraphical regular representation}). A DRR which is a graph
is called a \textit{GRR} (for \textit{graphical regular representation}).
DRRs and GRRs have been widely studied~\cite{babai1,Godsil,IW,nowitz1,nowitz2,nowitz3,nowitz4,W2,W3}, together with their friends: ORRs~\cite{morrisspiga,morrisspiga1,spiga} (oriented regular representations), TRRs~\cite{babai2} (tournament regular representations) and DFRs/GFRs~\cite{WT1,spiga0,spiga1} (digraphical/graphical Frobenius representations).

The aim of  this paper is to formulate ``running conjectures'' for the classification of the  bipartite DRRs and  bipartite GRRs and to make some progress towards these conjectures. We start by discussing some motivation for embarking into this task. The main techniques developed in the DRR and GRR classification (see for instance~\cite{babai1,Godsil,Hetzel}) involve a local analysis on the neighbourhood of a Cayley (di)graph. The easiest instance of this situation is probably in~\cite{babai1}; in this paper, for most finite groups $R$, Babai constructs a subset $S$ of $R$ with the property that $\Cay(R,S)$ is connected and with the property that the subgraph induced by $\Cay(R,S)$ on the neighbourhood of the identity is asymmetric. These two conditions imply that $\Cay(R,S)$ is a DRR, see for instance~\cite{nowitz1}. Then Babai is left to deal with the exceptional groups where his subset $S$ cannot be constructed. Broadly speaking, the same idea is constant throughout the investigation of GRRs, TRRs and ORRs. The investigation of DFRs and GFRs does not follow this pattern and require more algebraic tools. Similarly, the classification of bipartite DRRs and GRRs cannot follow this idea because the subgraph induced by $\Cay(R,S)$ on the neighbourhood of the identity is the empty graph, which brings no information.  Therefore, besides the natural interest in our opinion in classifying bipartite DRRs and bipartite GRRs (after all they are natural classes of graphs), we propose this problem for testing the new techniques that have been developed over the last 50 years in the study of graphical representations of groups.

After a long series of papers, Godsil~\cite{Godsil} has proved that, besides an explicit  list of exceptions, a group $R$ admits a GRR if and only if $R$ does not admit a non-identity automorphism $\varphi$ with $g^\varphi\in \{g,g^{-1}\}$, for every $g\in R$. Clearly, groups admitting such an automorphism $\varphi$ cannot admit a GRR because $\varphi$ is a non-identity automorphism for every Cayley graph $\Cay(R,S)$ over $R$.  Therefore, having a non-identity automorphism $\varphi$ with $g^\varphi\in \{g,g^{-1}\}$, for every $g\in R$, is a natural obstruction for a finite group $R$ for having a GRR. The GRR classification shows that this is indeed the only obstruction, besides a short list of small groups. It turns out that the only groups $R$ admitting such an automorphism are abelian groups of exponent greater than $2$ and generalised dicyclic groups. We believe that the same pattern holds for bipartite GRRs.

Let $R$ be a finite group and let $M$ be a subgroup of $R$ having index $2$. We say that the pair $(R,M)$ admits a bipartite GRR if there exists an inverse-closed subset $S$ of $R\setminus M$ with $R=\Aut(\Cay(R,S))$.  If $R$ has a non-identity automorphism $\varphi$ with $g^\varphi\in \{g,g^{-1}\}$, for every $g\in R\setminus M$, then $(R,M)$ cannot admit a bipartite GRR  because $\varphi$ is a non-identity automorphism for every bipartite Cayley graph $\Cay(R,S)$ over $R$ having bipartition $\{M,R\setminus M\}$. In particular, this is a natural obstruction for $(R,M)$ for admitting a bipartite GRR.  The first main result of our paper classifies the pairs $(R,M)$ satisfying this obstruction and hence not admitting a bipartite GRR. (We refer to Notation~\ref{notnot:1} for unexplained notation or terminology.)

\begin{theorem}\label{thrm:11}
Let $R$ be a finite group and let $M$ be a subgroup of $R$ with $|R:M|=2$. There exists a non-identity automorphism $\varphi$ of $R$ with $g^\varphi\in \{g,g^{-1}\}$, for every $g\in R\setminus M$, if and only if one of the following holds:
\begin{enumerate}
\item\label{thrmeq:1} $M$ is abelian and $R$ is not generalized dihedral on $M$;
\item\label{thrmeq:2} $M$  contains an abelian subgroup $Z$ with $|M:Z|=2$ and there exists $a\in R\setminus M$ with $a^2\ne 1$, $a^2\in Z\cap \Zent R$ and $z^a=z^{-1}$, for every $z\in Z$.
\item\label{thrmeq:3}$|M:\Zent M|=4$ and $\gamma_2(M)=\langle a^2\rangle$ for some $a\in R\setminus M$  such that $o(a)=4$, $z^a=z^{-1}$ for every $z\in \Zent M$, and $o(am)\ne 2$ for some $m\in M\setminus \Zent M$.
\end{enumerate}
\end{theorem}

Based on some extensive computer computations (see Remark~\ref{remark1} for details on these computations) we dare to propose the following conjecture:
\begin{conjecture}\label{conj:1}{\rm
Let $R$ be a finite group with a subgroup $M$ having index $2$. Then one of the following holds:
\begin{enumerate}
\item $(R,M)$ admits a bipartite GRR, or
\item $(R,M)$ satisfies part~\eqref{thrmeq:1},~\eqref{thrmeq:2} or~\eqref{thrmeq:3} in Theorem~\ref{thrm:11}, or
\item $R$ is one of the groups listed in Table~\ref{table:2}.
\end{enumerate}}
\end{conjecture}
The first column in Table~\ref{table:2} gives the order of $R$ and the second column gives the number of $R$ in the database of \texttt{SmallGroups} in the computer algebra system \textsf{GAP}, version 4.7.6.

\begin{table}[!ht]
\begin{tabular}{|c|c|c|}\hline
Grp. Order& Grp. Name\\\hline
 4& 2\\
 6& 1 \\
 8& 3, 5\\
 10& 1 \\
 12& 4 \\
 14& 1 \\
 16& 7, 8, 11, 13, 14\\
 18& 4 \\
 20& 3, 4 \\
 24& 13\\
 32& 6, 7, 8,  11, 15, 22, 31, 34, 37\\
&  38, 43, 44, 46, 48, 49, 50, 51\\
 36&9, 13 \\
 48& 46 \\
 54& 10\\
64&197, 198,   199,  200, 210, 214, 220,  222\\
&  223,
  224,
  225,  229,  230,  232,  235,  237\\
&  238,
  239,
  240,
  244,
  245, 267\\\hline
\end{tabular}
\caption{Exceptional small groups not admitting a bipartite GRR\label{table:2}}
\end{table}

From Theorem~\ref{thrm:11} part~\eqref{thrmeq:1}, we see that when $M$ is abelian, $(R,M)$ has no bipartite GRR unless (possibly) when $R$ is generalised dihedral on
$M$. Next, in this paper we take a closer look at these pairs and we investigate the existence of bipartite DRRs.

\begin{theorem}\label{thrm:1}Let $R$ be a finite group  with an abelian subgroup $M$ having index $2$. The number of subsets $S$ of $R\setminus  M$ such that $\Cay(R,S)$  is a bipartite $\mathrm{DRR}$  is at least $$2^{\frac{|R|}{2}}-5\cdot 2^{\frac{3|R|}{8}+\log_2|R|\cdot (\log_2(|R|/2))}.$$
\end{theorem}

Since $R\setminus M$ has $2^{|R\setminus M|}=2^{\frac{|R|}{2}}$ subsets, from Theorem \ref{thrm:1} we immediately obtain the following corollary.

\begin{corollary}\label{cor:1}For every positive real number $\varepsilon>0$, there exists a natural number $n_\varepsilon$ such that, for every finite group $R$ of order at least $n_\varepsilon$ and for every abelian subgroup $M$ of $R$ having index $2$, we have
$$\frac{|\{S\mid S\subseteq R\setminus M, \Cay(R,S) \,\mathrm{is\,a\, DRR}\}|}{|\{S\mid S\subseteq R\setminus M\}|}\ge 1-\varepsilon.$$
\end{corollary}

\begin{corollary}\label{cor:2}Let $R$ be a finite group  with an abelian subgroup $M$ having index $2$. Then, either there exists a subset $S$ of $R\setminus M$ such that $\Cay(R,S)$ is a bipartite  $\mathrm{DRR}$  or $R$ is one of the $22$ groups in the second column of Table~$\ref{table:1}$ and $M$ is one of the groups in the third column in Table~$\ref{table:1}$ subject to being abelian.
\end{corollary}

Exactly as for the bipartite GRR pairs, based on the evidence provided by Corollary~\ref{cor:2} and on some extensive computer computations (see again Remark~\ref{remark1} for some details on these computations) we propose the following conjecture:
\begin{conjecture}\label{conj:2}{\rm
Let $R$ be a finite group with a subgroup $M$ having index $2$. Then one of the following holds:
\begin{enumerate}
\item $(R,M)$ admits a bipartite DRR, or
\item $R$ is one of the $22$ groups listed in Table~\ref{table:1}.
\end{enumerate}}
\end{conjecture}
\begin{table}[!ht]
\begin{tabular}{|c|c|c|}\hline
Order& Group $R$& Subgroup 
$M$\\\hline
4& $C_2\times C_2$&any subgroup 
\\
6& $D_3=\langle x,y\mid x^3=y^2=1,x^y=x^{-1}\rangle$&$\langle y\rangle$\\
8& $D_4=\langle x,y\mid x^4=y^2=1,x^y=x^{-1}\rangle$& $\langle x\rangle$\\
& $Q_8$&any subgroup\\
&$C_4\times C_2=\langle x\rangle\times\langle y\rangle$&$\langle x^2,y\rangle$\\
&$C_2\times C_2\times C_2$&any subgroup\\
10& $D_5=\langle x,y\mid x^5=y^2=1,x^y=x^{-1}\rangle$&$\langle x\rangle$\\
12& $\langle x,y\mid x^6=y^4=1, x^3=y^2, y^{-1}xy=x^{-1}\rangle$&$\langle x\rangle$\\
&$D_6=\langle x,y\mid x^6=y^2=1,x^y=x^{-1}\rangle$&$\langle x\rangle$\\
14& $D_7=\langle x,y\mid x^7=y^2=1,x^y=x^{-1}\rangle$&$\langle x\rangle$\\
16& $C_4\times C_2\times C_2$&any subgroup\\
& $D_4\times C_2=\langle x,y\rangle\times \langle z\rangle$&subgroups containing $\langle x^2,z\rangle$\\
&$Q_8\times C_2$&any subgroup\\
&$\langle x,y,z\mid x^2=y^4=z^4=1,y^x=y^{-1},z^x=z^{-1},y^2=z^2,y^z=y^{-1}\rangle$ &$Q_8\cong \langle y,z\rangle$\\
&$C_2\times C_2\times C_2\times C_2$&any subgroup\\
18& $\langle x,y,z\mid x^2=y^3=z^3=1,yz=zy,y^x=y^{-1},z^x=z^{-1}\rangle$&$\langle y,z\rangle$\\
& $\langle x,y\mid x^6=y^3=1,xy=yx\rangle$&$\langle x^2,y\rangle$\\
32&$C_4\times C_2\times C_2\times C_2=\langle x\rangle\times\langle y\rangle\times\langle z\rangle\times \langle t\rangle$&$\langle x^2,y,z,t\rangle$\\
&$D_4\times C_2\times C_2=\langle x,y\rangle\times \langle z\rangle\times\langle t\rangle$& $\langle x,z,t\rangle$\\
&$Q_8\times C_2\times C_2=\langle x,y\rangle\times \langle z\rangle\times\langle t\rangle$& subgroups containing $\langle x^2,z,t\rangle$\\
&$C_2\times C_2\times C_2\times C_2\times C_2$&any subgroup\\
64&$C_2\times C_2\times C_2\times C_2\times C_2\times C_2$&any subgroup\\\hline
\end{tabular}
\caption{Small groups not admitting a bipartite DRR\label{table:1}}
\end{table}
We conclude this introductory section observing that in our companion paper~\cite{DFS} we have studied the asymptotic enumeration of bipartite graphs over abelian groups $A$. When $A$ has exponent greater than $2$, $A$ cannot admit a bipartite GRR in view of Theorem~\ref{thrm:11}. The work in~\cite{DFS} shows that, when $A$ has exponent greater than $2$, most bipartite graphs over $A$ have Cayley index $2$.

\subsection{General comments}

\begin{notation}\label{notnot:1}{\rm Our notation is standard. Given a group $G$, we denote by $\Zent G$ the center of $G$ and by $\gamma_2(G)$ the commutator subgroup of $G$. Given $g\in G$, we write $o(g)$ for the order of the element $g$. Given an automorphism $\varphi$ of $G$, we write
\begin{align*}
\cent G\varphi&:=\{g\in G\mid g^\varphi=g\},\\
\centm G\varphi&:=\{g\in G\mid g^{\varphi}=g^{-1}\}.
\end{align*}

We say that a group $D$ is a generalised dihedral group on $A$, if $A$ is an abelian
subgroup of index $2$ in $D$ and there exists an involution $\iota \in D \setminus A$ with $a^\iota = a^{-1}$, for
every $a \in A$. Note that, in this case, $a^x = a^{-1}$, for every $a \in A$ and every $x \in D \setminus A$.}
\end{notation}
\begin{remark}\label{remark1}{\rm There is a fair amount of computer computations involved in this paper. These computations are rather time consuming but entirely naive. For every group $R$ with $|R|< 1\, 024$, we have determined the subgroups $M$ of $R$ having index $2$. Then, for each pair $(R,M)$, we have tried to construct a bipartite DRR for $R$ with bipartition $\{M,R\setminus M\}$. Our approach for doing this is rather naive: for each pair $(R,M)$, we have randomly selected $10\,000$ subsets $S$ of $R\setminus M$ and we have checked whether the Cayley digraph $\Cay(R,S)$ was indeed a DRR. (For most pairs, $10$ iterations were sufficient to witness a subset $S\subseteq R\setminus M$ with $\Cay(R,S)$ a DRR.) There were only $22$ groups $R$  with $R$ having  order less then $1\, 024$ that did not pass this test and the largest of these groups $R$ is the elementary abelian $2$-group of order $64$. For these $22$ exceptional groups, we have checked exhaustively all the subsets $S$ of $R\setminus M$ and we confirmed in each case that there is no bipartite DRR for $R$ with bipartition $\{M,R\setminus M\}$. These $22$ groups (together with the subgroups $M$) are in Table~\ref{table:1}.

Table~\ref{table:2} was determined in  a similar manner and the only difference is that we used Theorem~\ref{thrm:1} in our algorithm.
 For every group $R$ with $|R|< 512$(there should be 640 consistent with in the ``Proof of Corollary 1.5"), we have determined the subgroups $M$ of $R$ having index $2$. Then, we have discarded the pairs $(R,M)$ satisfying parts~\eqref{thrmeq:1},~\eqref{thrmeq:2} or~\eqref{thrmeq:3} in Theorem~\ref{thrm:1}, because there exists no bipartite GRR over $R$ with bipartition $\{M,R\setminus M\}$. In light of Theorem~\ref{thrm:1} this task is pretty straighforward and fast to perform. For the remaining pairs, we have tried to construct a bipartite GRR for $R$ with bipartition $\{M,R\setminus M\}$. Our approach for doing this is as above:  we have randomly selected $10\,000$ inverse-closed subsets $S$ of $R\setminus M$ and we have checked whether the Cayley graph $\Cay(R,S)$ was indeed a GRR. (For most pairs, $500$ iterations were sufficient to witness an inverse-closed subset $S\subseteq R\setminus M$ with $\Cay(R,S)$ a GRR.) The groups $R$ that did not pass this test (for some subgroup $M$ having index $2$ and with $(R,M)$ not satisfying parts~\eqref{thrmeq:1},~\eqref{thrmeq:2} or~\eqref{thrmeq:3} in Theorem~\ref{thrm:1}) are reported in Table~\ref{table:2}. We are not reporting the subgroups $M$ in Table~\ref{table:2} because, for some groups $R$, there are many choices for $M$, which would make difficult to compile the table in a ready to use way. Finally, for these exceptional groups, we have checked exhaustively all the inverse-closed subsets $S$ of $R\setminus M$ and we confirmed in each case that there is no bipartite GRR for $R$ with bipartition $\{M,R\setminus M\}$.
}
\end{remark}

In what follows we use repeatedly the following facts.
\begin{facts}{\rm
\begin{enumerate}
\item\label{fact1} Let $X$ be a finite group. Since a chain of subgroups of $X$ has length at most $\lfloor\log_2|X|\rfloor$, $X$ has a generating set of cardinality at most $\lfloor \log_2|X|\rfloor\le \log_2|X|$.

\item\label{fact2} Any automorphism of $X$ is uniquely determined by the images of the elements of a generating set for $X$. Therefore $|\Aut(X)|\le |X|^{\lfloor \log_2|X|\rfloor}\le 2^{(\log_2|X|)^2}$.

\item\label{fact3} Any subgroup $Y$ of $X$ is determined by a generating set, which has cardinality at most $\lfloor \log_2|Y|\rfloor\le \lfloor \log_2|X|\rfloor$. Therefore $X$ has at most
$|X|^{\lfloor\log_2|Y|\rfloor}\le 2^{\log_2|Y|\log_2|X|}$ subgroups of cardinality $|Y|$ and $X$ has at most $2^{(\log_2|X|)^2}$ subgroups.
\end{enumerate}}
\end{facts}

\section{
route to the proof of Theorem~\ref{thrm:1} and Corollaries~\ref{cor:1} and~\ref{cor:2}}\label{part:1}
\begin{lemma}\label{l:3}
Let $R$ be a group and let $M$ be a subgroup of $R$ having index $2$. The number of subsets $S$ of $R\setminus M$ with $\langle S\rangle $ a proper subgroup of $R$ is at most $2^{\frac{|R|}{4}+(\log_2|R|)^2}$; moreover, when $R$ is solvable, this upper bound can be improved to  $2^{\frac{|R|}{4}+\log_2|R|}$.
\end{lemma}
\begin{proof}
Set $N:=|\{
S\subseteq R\setminus  M\mid
\langle S\rangle < R\}|$.
Clearly,
\begin{align*}
\{S\subseteq R\setminus M\mid \langle S\rangle<R\}&=\bigcup_{\substack{C<M\\C\textrm{ maximal in }M}}\{S\subseteq R\setminus M\mid \langle S\rangle\le C \}.
\end{align*}
Since $\{S\subseteq R\setminus M\mid \langle S\rangle \le C\}=\{S\mid S\subseteq C\setminus (C\cap  M)\}$, we have
$$|\{S\subseteq R\setminus M\mid \langle S\rangle \le C\}|=2^{|C|-|C\cap M|}\le 2^{\frac{|C|}{2}}\le 2^{\frac{|R|/2}{2}}=2^{\frac{|R|}{4}}.$$
Therefore
\begin{align*}
N&
\le |\{C\le G\mid C \textrm{ is a maximal subgroup of }G\}|\cdot 2^{\frac{|R|}{4}}.
\end{align*}
Tim Wall in 1961~\cite{Wall} has proved that the number of maximal subgroups of a finite solvable  group $R$ is less than the group order $|R|=2^{\log_2|R|}$. In particular, our proof is completed in the case of solvable groups. Whereas the general bound $2^{|R|/4+(\log_2|R|)^2}$ follows from Fact~\eqref{fact3}.
\end{proof}

Liebeck, Pyber and Shalev have proven~\cite[Theorem~1.3]{LPSLPS} a polynomial version of Wall's theorem for arbitrary finite groups: there exists an absolute constant $c$ such that, every finite group $R$ has at most $c|R|^{3/2}=2^{\log_2(c|R|^{3/2}) }$ maximal subgroups. Hence, if one minds so, the general bound $2^{\frac{|R|}{4}+(\log_2|R|)^2}$ can be improved to $2^{\frac{|R|}{4}+\log_2(c|R|^3/2)}$, for some absolute constant $c$.

\begin{lemma}\label{l:1}Let $R$ be a finite group with an abelian subgroup $M$ having index $2$ and let $\varphi$ be a non-identity automorphism of $R$ with $M^\varphi=M$. The number of subsets $S$ of $R\setminus M$ with $S^\varphi=S$ is at most $2^{\frac{3|R|}{8}}$.
\end{lemma}
\begin{proof}
Since $M$ is $\varphi$-invariant, so is $R\setminus M$. Let $O_1,\ldots,O_\ell$ be the orbits of $\langle\varphi\rangle$ on $R\setminus M$. If $S\subseteq R\setminus M$ is $\varphi$-invariant, then $S$ is a union of some of $O_1,\ldots,O_\ell$ and hence
\begin{equation}\label{eq:1vai}
|\{S\subseteq R\setminus M\mid S^\varphi=S\}|=2^\ell.
\end{equation}

The orbits of $\langle\varphi\rangle$ on $R$ of cardinality one correspond exactly to the elements of $\cent R{\varphi}:=\{a\in R\mid a^\varphi=a\}$, whereas the orbits of $\langle\varphi\rangle$ on $R\setminus \cent R \varphi$ have cardinality at least $2$.  Now, observing that $|\cent R\varphi|\le |R|/2$ and that
\[
|\cent R\varphi\cap(R\setminus M)|=
\begin{cases}
0&\textrm{ when }\cent R\varphi\le M,\\
|\cent R \varphi \cap M|=|\cent R\varphi|/2&\textrm{ when }\cent R\varphi\nleq M,
\end{cases}
\] we get
\begin{align}\label{eq:2vai}
\ell&\le
|\cent R{\varphi}\cap (R\setminus M)|
+
\frac{
|(R\setminus M)\setminus (\cent R{\varphi}\cap (R\setminus M))|}{2}
=\frac{|\cent R{\varphi}\cap (R\setminus M)|}{2}+\frac{|R\setminus M|}{2}\\\nonumber
&= \frac{|\cent R{\varphi}\cap (R\setminus M)|}{2}+\frac{|R|}{4}\le \frac{|\cent R\varphi|}{4}+\frac{|R|}{4}\le \frac{|R|/2}{4}+\frac{|R|}{4}=\frac{3}{8}|R|.
\end{align}

The proof now follows from \eqref{eq:1vai} and \eqref{eq:2vai}.
\end{proof}

The next lemma is somehow more technical but relevant for the proof of Theorem~\ref{thrm:1}.
\begin{lemma}\label{l:4}Let $R$ be a finite group with an abelian subgroup $M$ having index $2$. Then the number of subsets $S\subseteq R\setminus M$ such that
\begin{itemize}
\item $\Cay(R,S)$ is connected, and
\item $R<\Aut(\Cay(R,S))$, and
\item  $\nor{\Aut(\Cay(R,S))}R=R$, and
\item $R<\nor{\Aut(\Cay(R,S))}M$
\end{itemize}
 is at most $2^{\frac{|R|}{3}+\log_2(|R|/2)+(\log_2(|R|/2))^2}$.
\end{lemma}
\begin{proof}
Let $S$ be a subset of $R\setminus M$ and suppose that $R<\Aut(\Cay(R,S))$, $\nor {\Aut(\Cay(R,S))}R=R$ and $R<\nor {\Aut(\Cay(R,S))}M$. Choose $G\le \nor{\Aut(\Cay(R,S))}{M}$ with $R<G$ and with $R$ maximal in $G$: this is of course possible because $R<\nor{\Aut(\Cay(R,S))}{M}$. Observe that $\nor GM=G$ and $\nor GR=R$.

Fix $\varphi\in G_1\setminus R$. As $\nor G R=R$, we have $R^\varphi\ne R$. Since $R$ is maximal in $G$, we have $G=\langle R,R^\varphi\rangle$. Observe that in $G/M=\langle R/M,R^\varphi/M\rangle$, the groups $R/M$ and $R^\varphi/M$ have order $2$ and hence $G/M$ is a dihedral group. Since $R/M$ is maximal in $G/M$, we deduce that $G/M$ is a dihedral group of order $2p$, for some odd prime number $p$. Since $G=RG_1$, we deduce that $G_1=\langle \varphi\rangle$ is cyclic of prime order $p\ge 3$.

We now set some notation. The elements in $R$ can simultaneously represent the vertices of the digraph $\Cay(R,S)$ as well as the translation automorphisms. Given $d\in R$, we denote by $\varphi^{-1}d\varphi$ the automorphism of $\Cay(R,S)$ obtained by applying first $\varphi$, then the right translation by $d$ and then $\varphi$. Whereas, we denote by $d^\varphi$ the image of the vertex $d$ under the automorphism $\varphi$. This notation is consistent with the work of Godsil in this area, see~\cite{Godsil}.

Consider $a\in M$. Then $\varphi^{-1}a\varphi\in M$ because $\varphi$ normalises $M$. Since $\Cay(R,S)$ is connected, $\{M,R\setminus M\}$ is the only bipartion of $\Cay(R,S)$ and hence $a^\varphi\in M$. Moreover, $1^{\varphi^{-1} a\varphi}=(1^{\varphi^{-1}})^{a\varphi}=1^{a\varphi}=a^\varphi=1^{a^\varphi}$. Therefore, $\varphi^{-1}a\varphi$ and $a^\varphi$ are two elements of $M$ mapping the vertex $1$ to the same vertex. Since $M$ acts semiregularly, these two elements must be equal and hence $$\varphi^{-1}a\varphi=a^\varphi,\quad\forall a\in M.$$

Fix $d\in R\setminus M$. For every $a\in M$, we have
\begin{equation}\label{eq:silly}
(da)^\varphi=1^{da\varphi}=1^{d\varphi \varphi^{-1}a\varphi}=(1^{d\varphi})^{\varphi^{-1}a\varphi}=(1^{d\varphi})a^\varphi=d^\varphi a^\varphi.
\end{equation}
This means that the mapping $\varphi:R\to R$ is uniquely determined by the image of $d$ and by the restriction $\varphi_{|M}$ of $\varphi$ to $M$. Since we have $|R\setminus M|=|M|$ choices for the image of $d$ and since we have $|\Aut(M)|$ choices for $\varphi_{|M}$, we have at most $|M||\Aut(M)|\le 2^{\log_2|M|+(\log_2|M|)^2}$ choices for $\varphi$.

Let us now count the number of subsets $S\subseteq R\setminus M$ invariant by $\varphi$. Clearly, $S$ is a union of $\langle\varphi\rangle$-orbits. Suppose first that  $\varphi$ has no fixed points on $R\setminus M$. Since $\varphi$ has prime order $p\ge 3$, we obtain that each orbit of $\langle\varphi\rangle$ on $R\setminus M$ has cardinality at least $3$ and hence the number of choices for $S$ is at most
$$2^{\frac{|R\setminus M|}{3}}=2^{\frac{|R|}{6}}.$$
Suppose now that $\varphi$ fixes some point in $R\setminus M$. Without loss of generality, we may suppose that $d$ is fixed by $\varphi$ and hence $d^\varphi=d$. Now,~\eqref{eq:silly}  gives $(da)^\varphi=da^\varphi,$
which shows that $da$ is fixed by $\varphi$ if and only if $a\in \cent M \varphi$.  If $\varphi$ centralises $M$, then $\varphi$ fixes each vertex of $\Cay(R,S)$, contrary to our assumption that $\varphi\ne 1$. Thus $\cent M\varphi\ne M$ and hence $|M:\cent M\varphi|\ge 2$. Therefore the number of choices for $S$ is at most
$$2^{|\cent M \varphi|+\frac{|M\setminus \cent M\varphi|}{3}}=2^{\frac{|M|}{3}+\frac{2|\cent M\varphi|}{3}}\le \
2^{\frac{|M|}{3}+\frac{2(|M|/2)}{3}}
\le 2^{\frac{|M|}{3}+\frac{|M|}{3}}=2^{\frac{|R|}{3}}.$$
Now the proof follows.
\end{proof}

\begin{lemma}\label{l:5}Let $R$ be a finite group with an abelian subgroup $M$ having index $2$. The number of subsets $S\subseteq R\setminus M$ such that the stabiliser in $\Aut(\Cay(R,S))$ of the bipartition $\{M,R\setminus M\}$ does not act faithfully on $M$ is at most $2^{\frac{|R|}{4}+(\log_2(|R|/2))^2}$.
\end{lemma}
\begin{proof}
Let $S$ be a subset of $R\setminus M$ and suppose that the stabiliser in $\Aut(\Cay(R,S))$ of the bipartition $\{M,R\setminus M\}$ does not act faithfully on $M$.  For simplicity we write $G:=\Aut(\Cay(R,S))$ and we denote by $G^+$ the stabiliser of the bipartition $\{M,R\setminus M\}$ of $\Cay(R,S)$. Clearly, $|G:G^+|=2$.

Let $K_1$ be the kernel of the action of $G$ on $M$ and let $K_2$ be the kernel of the action of $G$ on $R\setminus M$. Set $K:=K_1\times K_2$. Clearly, $K\unlhd G$ and by hypothesis $K\ne 1$. Now define $H:=RK$ and observe that $H\ne R$ because $R$ acts regularly on the graph, but $K$ has non-identity permutations fixing some vertex.

Since $K\unlhd H$, the orbits of $K$ form a system of imprimitivity for the action of $H$. This system of imprimitivity is also a system of imprimitivity for $R$ because $R\le H$. Since $R$ acts regularly, this system of imprimitivity consists of the cosets of a certain subgroup $L$ of $R$. In particular, $d^K=dL$, for every $d\in R$. We have $L\le M$ because $K$ preserves the bipartition $\{M,R\setminus M\}$ of the graph. Moreover, $L\ne 1$ because $K\ne 1$.

We claim that $S$ is a union of $L$-cosets. To this end, it suffices to show that, for every $s\in S$, we have $sL\subseteq S$. Let $s\in S\subseteq R\setminus M$. Since the stabiliser $H_1$ of the vertex $1$  contains  $K_2$ and since $H$ acts as a group of automorphisms of $\Cay(R,S)$, we have $s^{K_2}\subseteq s^{H_1}\subseteq S$. Since the elements in $K_1$ fix $M$ pointwise, $s^{K_2}=s^{K_1\times K_2}=s^{K}=sL$.

From the previous paragraph, when the subgroup $L$ of $M$ is given, the number of choices for $S$ is at most
$$2^{\frac{|R\setminus M|}{|L|}}\le 2^{\frac{|R|}{4}},$$
because $|L|\ge 2$. The number of choices of $L$ is at most $|M|^{\log_2|M|}$ by Fact~\eqref{fact3}.
\end{proof}

\begin{proof}[Proof of Theorem~$\ref{thrm:1}$]
Let $R$ be a finite group with an abelian subgroup $M$ having index $2$ and let
\begin{equation*}
\mathcal{S}:=\{S\subseteq R\setminus M\mid \Cay(R,S) \textrm{ is not a DRR}\}.
\end{equation*}
We partition the set $\mathcal{S}$ in various subsets. First,
\begin{equation*}
\mathcal{S}_1:=\{S\in\mathcal{S}\mid \Cay(R,S) \textrm{ is not connected}\}.
\end{equation*}
By Lemma~\ref{l:3}, we have
\begin{equation}\label{eq1:1}
|\mathcal{S}_1|\le 2^{\frac{|R|}{4}+\log_2|R|}.
\end{equation}
Observe that, if $S\in\mathcal{S}\setminus \mathcal{S}_1$, then $\Cay(R,S)$ is connected  and hence $\{M,R\setminus M\}$ is the only bipartition of $\Cay(R,S)$. In particular, every automorphism of $\Cay(R,S)$ either fixes $M$ setwise, or maps $M$ to $R\setminus M$.

Then, we set
\begin{equation*}
\mathcal{S}_2:=\{S\in\mathcal{S}\setminus \mathcal{S}_1\mid \nor{\Aut(\Cay(R,S))}R>R\}.
\end{equation*}
If $S\in\mathcal{S}_2$, then there exists a non-identity automorphism $\varphi\in \nor{\Aut(\Cay(R,S))}R$ with $1^\varphi=1$. Clearly, $M^\varphi=M$ and $\varphi$ is an automorphism of $R$. By Lemma~\ref{l:1}, we have at most $2^{3|R|/8}$ choices for $S$ when $\varphi$ is fixed. Fix $d\in R\setminus M$. The automorphism $\varphi$ of $R$ is uniquely determined by the image of $d$ (which is an element of $R\setminus M$ because $M^\varphi=M$) and by the restriction $\varphi_{|M}$ of $\varphi$ to $M$. Therefore, we have at most $|M||\Aut(M)|$ choices for $\varphi$. By Fact~\eqref{fact2}, we have
$$|M||\Aut(M)|\le 2^{\log_2|M|+(\log_2|M|)^2}=2^{(\log_2|R|)(\log_2(|R|/2))}.$$
It follows that
\begin{equation}\label{eq1:2}
|\mathcal{S}_2|\le 2^{\frac{3|R|}{8}+(\log_2|R|)(\log_2(|R|/2))}.
\end{equation}
For every $S\in\mathcal{S}\setminus (\mathcal{S}_1\cup\mathcal{S}_2)$, $\Cay(R,S)$ is connected and $R$ is self-normalising in $\Aut(\Cay(R,S))$.

Then, we set
\begin{equation*}
\mathcal{S}_3:=\{S\in\mathcal{S}\setminus (\mathcal{S}_1\cup \mathcal{S}_2)\mid \nor{\Aut(\Cay(R,S))}M>R\}.
\end{equation*}
By Lemma~\ref{l:4}, we have
\begin{equation}\label{eq1:3}
|\mathcal{S}_3|\le 2^{\frac{|R|}{3}+\log_2(|R|/2)+(\log_2(|R|/2))^2}.
\end{equation}
Now, for every $S\in\mathcal{S}\setminus (\mathcal{S}_1\cup\mathcal{S}_2\cup\mathcal{S}_3)$, we have that $\Cay(R,S)$ is connected, $R$ is self-normalising in $\Aut(\Cay(R,S))$ and the normaliser of $M$ in $\Aut(\Cay(R,S))$ is $R$.

We set
\begin{align*}
\mathcal{S}_4:=\{S\in\mathcal{S}\setminus (\mathcal{S}_1\cup \mathcal{S}_2\cup \mathcal{S}_3)\mid& \textrm{ the stabilizer in }\Aut(\Cay(R,S)) \textrm{ of the bipartition}\\
&\{M,R\setminus M\} \textrm{ does not act faithfully on }M\}.
\end{align*}
By Lemma~\ref{l:5}, we have
\begin{equation}\label{eq1:4}
|\mathcal{S}_4|\le 2^{\frac{|R|}{4}+(\log_2(|R|/2))^2}.
\end{equation}

We set
\begin{align*}
\mathcal{S}_5:=\mathcal{S}\setminus (\mathcal{S}_1\cup \mathcal{S}_2\cup \mathcal{S}_3\cup \mathcal{S}_4).
\end{align*}
Now, for every $S\in\mathcal{S}_5$, we have that $\Cay(R,S)$ is connected, $R$ is self-normalising in $\Aut(\Cay(R,S))$, the normaliser of $M$ in $\Aut(\Cay(R,S))$ is $R$ and the stabiliser in $\Aut(\Cay(R,S))$ of the bipartition $\{M,R\setminus M\}$ acts faithfully on $M$ (and hence also on $R\setminus M$).

For every $S\in\mathcal{S}\setminus (\mathcal{S}_1\cup\mathcal{S}_2\cup\mathcal{S}_3\cup\mathcal{S}_4)$, choose a subgroup $G^+$ of the stabiliser in $\Aut(\Cay(R,S))$ of the bipartition $\{M,R\setminus M\}$ with $M$ maximal in $G^+$.  Strictly speaking, we should use a notation for $G^+$ witnessing that it depends on $S$, but for not making the notation too cumbersome to use we avoid that.  Observe that $\nor {G^+}M=M$ and that $G^+$ acts faithfully on $M$ and on $R\setminus M$.
In particular, we are in the position to apply~\cite[Theorem~$3.2$]{DSV}, which we report below with our current notation.

\begin{theorem}{{~\cite[Theorem~$3.2$]{DSV}}}\label{thm:main}
Let $G^+$ be a permutation group on $\Delta$ with a maximal abelian regular subgroup $M$ such that $\nor {G^+} M = M$. Let $(G^+)_\delta$ be the stabiliser of the point $\delta\in \Delta$, let $N$ be the core of $M$ in $G^+$. Then there exist a prime $p$ and $Q$ and $W$ with $Q\neq 1\neq W$ such that
\begin{enumerate}
\item $G^{+}/N\cong (G^+)_\delta N/N \rtimes M/N$ acts faithfully as an affine primitive group on the cosets of $M$ in $G^{+}$, \label{yam}
\item $N=\Z {G^+}=Q\times \cent W {(G^+)_\delta}$, \label{yi}
\item $(G^+)_\delta\times Q$ is the unique Sylow $p$-subgroup of $G^+$, \label{yu}
\item  $\nor {G^+} {(G^+)_\delta}=\cent {G^+} {(G^+)_\delta}=(G^+)_\delta\times N$, \label{yog}
\item for all $s,s'\in G^{+}\setminus \nor {G^+} {(G^+)_\delta}$, we have $(G^+)_\delta (G^+)_{\delta^s}=(G^+)_\delta (G^+)_{\delta^{s'}}$. \label{youm}
\end{enumerate}
\end{theorem}

Consistent with the notation in Theorem~\ref{thm:main}, we let $N$ be the core of  $M$ in $G^+$. We now partition the set $\mathcal{S}_5$ in two subsets:
\begin{align*}
\mathcal{S}_5'&:=\{S\in\mathcal{S}_5\mid (G^+)_1=(G^+)_d, \textrm{ for some }d\in R\setminus M\},\\
\mathcal{S}_5''&:=\mathcal{S}_5\setminus \mathcal{S}_5'.
\end{align*}

\smallskip

\noindent\textsc{Claim: }Fix $d\in R\setminus M$. For every $S\in\mathcal{S}_5'$, there exists a maximal subgroup $K'$ of $M$ and a minimal subgroup $H'$ of $K'$ such that $S\setminus dK'$ is a union of $H'$-cosets.

\smallskip

\noindent
Let $S\in \mathcal{S}_5'$, that is, $(G^+)_1=(G^+)_d$, for some $d\in R\setminus M$.
We apply Theorem~\ref{thm:main} for the action of $G^+$ on $\Delta:=R\setminus M$ and with $\delta:=d$. Observe that from Theorem~\ref{thm:main}, $(G^+)_d$ is a $p$-group and hence $(G^+)_x$ is a $p$-group for every $x\in R$ because $|(G^+)_d|=|(G^+)_x|$.

Let $T:=\nor {G^+}{(G^+)_d}$. By Theorem~\ref{thm:main}~(\ref{yi}), (\ref{yu}) and~(\ref{yog}), $T$ contains the unique Sylow $p$-subgroup of $G^+$ and hence $(G^+)_y\leq T$ for every $y\in \Delta=R\setminus M$. Since $(G^+)_d$ is normal in $T$, it follows that $(G^+)_d(G^+)_y=(G^+)_d(G^+)_y$ is a subgroup of $T$ and
\begin{equation}\label{eq:tofu}(G^+)_1(G^+)_y=(G^+)_d(G^+)_y=(G^+)_y(G^+)_d=(G^+)_y(G^+)_1.
\end{equation}
Let $s\in G^+\setminus T$ and let $$H:=(G^+)_1(G^+)_{d^s}\cap M.$$ By Theorem~\ref{thm:main}~(\ref{youm}) and~\eqref{eq:tofu}, $H$ does not depend on the choice of $s$. Assume $H=1$. As $M$ acts regularly on $M$ and on $R\setminus M$, we have $G^+=M(G^+)_1=M(G^+)_{d^s}$. Therefore, for every $x\in (G^+)_1$, there exists $a\in M$ and $y\in (G^+)_{d^s}$ with $x=ay$. Thus $xy^{-1}=a\in (G^+)_1(G^+)_{d^s}\cap M=H=1$ and hence $x=y$. Since $x$ is an arbitrary element in $(G^+)_1$, this yields $(G^+)_1=(G^+)_{d^s}$, that is, $(G^+)_d=(G^+)_1=(G^+)_{d^s}=((G^+)_d)^s$  and $s\in \nor {G^+}{(G^+)_d}=T$, which is a contradiction. Therefore $H\neq 1$.

Let $K:=N$. By Theorem~\ref{thm:main}~(\ref{yog}), $T\cap M = ((G^+)_d\times N)\cap M=((G^+)_d\cap M)\times N = N= K$ and hence $H\leq K<M$.

Let $x$ in $M\setminus K$. Since $T\cap M=K$, we have $x\notin T$ and $H=(G^+)_1(G^+)_{dx}\cap M$. Since $(G^+)_1(G^+)_{dx}=(G^+)_d(G^+)_{dx}$ is a subgroup containing $(G^+)_d$, it follows that $(dx)^{(G^+)_1(G^+)_{dx}}$ is a block of imprimitivity for $G^+$ and hence also for $M$. Moreover, $(G^+)_1(G^+)_{dx}$ is the stabiliser of this block in $G^+$, hence $(G^+)_1(G^+)_{dx}\cap M=H$ is the stabiliser of this block in $M$, therefore $(dx)^{(G^+)_1(G^+)_{dx}}$ is an $H$-coset, that is, $(dx)^{(G^+)_1(G^+)_{dx}}=dxH$. On the other hand, $(dx)^{(G^+)_1}=(dx)^{(G^+)_{dx}(G^+)_1}=(dx)^{(G^+)_1(G^+)_{dx}}=dxH$. We have shown that every $(G^+)_1$-orbit on $\Delta\setminus dK=(R\setminus M)\setminus dK$ is an $H$-coset.  In particular, our set $S\setminus dK$ is a union of $H$-cosets.

Fix now a maximal subgroup $K'$ of $M$ with $K\le K'$ and $H'$ a minimal subgroup of $H$ with $H'\le H$. Since $S\setminus dK$ is a union of $H$-cosets,  $S\setminus dK'$ is a union of $H'$-cosets.~$_\blacksquare$

\smallskip

Given $H'$ and $K'$ subgroups of $M$ with $1<H'\le K'<M$, the number of subsets $S$ of $R\setminus M$ such that $S\setminus dK'$ is a union of $H'$-cosets is at most
\begin{equation*}
2^{|K|+\frac{|M\setminus K|}{|H|}}=2^{\frac{|M|}{|H|}+|K|\left(1-\frac{1}{|H|}\right)}\le 2^{\frac{|M|}{|H|}+\frac{|M|}{2}\left(1-\frac{1}{|H|}\right)}=2^{\frac{|M|}{2}+\frac{|M|}{2|H|}}\le 2^{\frac{|M|}{2}+\frac{|M|}{4}}=2^{\frac{3|R|}{8}}.
\end{equation*}
The number of maximal subgroups $K'$ of an abelian group $M$ is at most $|M|$ and the number of minimal subgroups $H'$ of a group $K'$ is at most $|K'|\le |M|/2$.
 In particular, from Claim, we deduce
\begin{equation}\label{lastlast}|\mathcal{S}_5'|\le 2^{\frac{3|R|}{8}+\log_2(|R|/2)+\log_2(|R|/4)}=2^{\frac{3|R|}{8}+2\log_2|R|-3}.
\end{equation}

\smallskip

It remains to estimate the cardinality of the set $\mathcal{S}_5''$.
Given a finite group $X$, we write $$f(X):=|\{Y\le X\mid |Y| \textrm{ is prime}\}|$$
and similarly, given a positive integer $n$, we write
$$f(n):=\max\{f(X)\mid X\textrm{ is a group of order }n\}.$$

Let $S\in\mathcal{S}_5''$, that is, there is no $d\in R\setminus M$ with $(G^+)_1= (G^+)_d$. Let $d\in R\setminus M$. From Theorem~\ref{thm:main}~\eqref{yam} and~\eqref{yog}, we see that $(G^+)_d\times N$ is a normal subgroup of $G^+$, from which it follows that $G^+$ fixes setwise each $N$-orbit on $R\setminus M$.

  Let $d\in R\setminus M$. From Theorem~\ref{thm:main}~\eqref{yu}, $(G^+)_d\times Q$ is the unique Sylow $p$-subgroup of $G^+$. As $(G^+)_1$ is a $p$-group,  $(G^+)_1$ is contained in $(G^+)_d\times Q\le (G^+)_d\times N$. In particular, $(G^+)_1$ fixes each $N$-orbit.

For each $d\in R\setminus M$, the action induced by $(G^+)_d\times N$ on $d^N=dN$ is regular (given by the subgroup $N$) and hence the action induced by $(G^+)_1$ on $d^N$ is semiregular and given by some subgroup $X(d)$ of $N$, because $(G^+)_1\le (G^+)_d\times N$. (Observe that $X(d)$ depends on the vertex $d$.) Since $(G^+)_1\ne (G^+)_d$, for every $d\in R\setminus M$,  $(G^+)_1$ has no fixed point on $R\setminus M$ and hence $X(d)\ne 1$.

Let $d_1,\ldots,d_{|M:N|}$ be the representatives of the $N$-orbits on $R\setminus M$ and, for every $i\in \{1,\ldots,|M:N|\}$, let $Y(d_i)$ be any subgroup of  $X(d_i)$ having prime order.

Since the set $S$ is $(G^+)_1$-invariant, $S\cap d_i^{N}$ is $X(d_i)$-invariant, for every $i\in \{1,\ldots,|M:N|\}$. In particular, $S\cap d_i^N$ is also $Y(d_i)$-invariant. When the subgroup $Y(d_i)$ of $N$ is given, the number of possibilities for $S\cap d_i^{N}$ is at most $2^{|N|/|Y(d_i)|}\le 2^{|N|/2}$, because $|Y(d_i)|\ne 1$.

When the subgroup $N$ of $M$ is given, we have at most $f(N)2^{|N|/2}\le f(|N|)2^{|N|/2}$ choices for $S\cap d_i^N$. In particular, when the subgroup $N$ of $M$ is given, since with have $|M|/|N|$ choices for $i$, the number  of choices for $S$ is at most
$$(f(|N|)2^{|N|/2})^{|M|/|N|}=2^{|M|\frac{\log_2(f(|N|))}{|N|}+\frac{|M|}{2}}.$$
Given a divisor $n$ of $|M|$, the number of subgroups $N$ of $M$ having order $n$ is at most $|M|^{\log_2|n|}$ by Fact~\eqref{fact3}. Therefore,
\begin{equation}\label{eq:last}
|\mathcal{S}_5''|\le \sum_{\substack{n\mid |M|\\ 1<n<|M|}}2^{|M|\frac{\log_2(f(n))}{n}+\frac{|M|}{2}+\log_2|M|\log_2n}.
\end{equation}
Observe now that $f(p)=1$, for every prime number $p$. In general, $f(n)\le n-1$, for every $n$. We consider the auxiliary real function $F:[2,|M|/2]\to \mathbb{R}$ defined by $$x\mapsto F(x):=\frac{\log_2(x-1)}{x}|M|+\log_2 x\log_2|M|.$$
We have
\begin{align*}
F'(x)&=\left(\frac{1}{(x-1)x}-\frac{\log(x-1)}{x^2}\right)\frac{|M|}{\log 2}+\frac{\log_2 |M|}{\log 2\cdot x}\ge
\left(\frac{1}{(x-1)x}-\frac{\log(x-1)}{x^2}\right)\frac{\log_2|M|}{\log 2}+\frac{\log_2 |M|}{\log 2\cdot x}\\
&=\frac{\log_2|M|}{\log 2\cdot x^2}\left(\frac{x}{x-1}-\log(x-1)+x\right)\ge 0,
\end{align*}
where the last inequality follows with an easy computation. In particular, $F(x)$ is an increasing function of $x$ and hence
\begin{eqnarray*}
\frac{\log_2(f(n))}{n}|M|+\log_2|M|\log_2n&\le& 2\log_2(|M|/2)+\log_2|M|\log_2(|M|/2)\\
&=&\log_2|M|-2+(\log_2|M|)^2.
\end{eqnarray*}
From this and from~\eqref{eq:last}, we deduce
\begin{eqnarray}\label{lastlastlast}
|\mathcal{S}_5''|&\le& \sum_{\substack{n\mid |M|\\1<n<|M|}}2^{\frac{|M|}{2}+\log_2|M|-2+(\log_2|M|)^2}
\le 2^{\frac{|M|}{2}+\log_2|M|-2+(\log_2|M|)^2}\cdot 2\sqrt{|M|}\\\nonumber
&=& 2^{\frac{|R|}{4}+\frac{3}{2}\log_2|R|-\frac{5}{2}+(\log_2(|R|/2))^2},
\end{eqnarray}
where the second inequality follows from the fact that $|M|$ has at most $2\sqrt{|M|}=2^{1+(\log_2|M|)/2}$ divisors.

In~\eqref{eq1:1},~\eqref{eq1:2},~\eqref{eq1:3},~\eqref{eq1:4},~\eqref{lastlast} and~\eqref{lastlastlast}, we have estimated the cardinalities of the sets $\mathcal{S}_1,\mathcal{S}_2,\mathcal{S}_3,\mathcal{S}_4,\mathcal{S}_5'$ and $\mathcal{S}_5''$. The main contribution comes from the estimate of $\mathcal{S}_2$ in~\eqref{eq1:2} and hence we obtain
$$|\mathcal{S}|\le 5\cdot 2^{\frac{3|R|}{8}+\log_2|R|\cdot (\log_2(|R|/2))}.~\qedhere$$
\end{proof}

\begin{proof}[Proof of Corollary~$\ref{cor:2}$]
Let $R$ be a finite group with an abelian subgroup $M$ having index $2$. If
$$2^{\frac{|R|}{2}}>5\cdot 2^{\frac{3|R|}{8}+\log_2|R|\cdot (\log_2(|R|/2))},$$
then by Theorem~\ref{thrm:1} there exists a subset $S$ of $R\setminus M$ with $\Cay(R,S)$ a DRR. A computation shows that this inequality holds whenever $|R|\ge 640$.

When $|R|< 640$(there should consistent with Remark~\ref{remark1}), the proof follows with a computer computation, see Remark~\ref{remark1}.
\end{proof}

\section{A
route to the proof of Theorem~\ref{thrm:11}}\label{enroute2}

\begin{lemma}\label{example:1}
Let $R$ be a group and let $M$ be an abelian subgroup of $R$ with $|R:M|=2$. Then, there exists a non-identity automorphism $\varphi$ of $R$ with $g^\varphi\in \{g,g^{-1}\}$, for every $g\in R\setminus M$, if and only if $R$ is not generalised dihedral on $M$.
\end{lemma}
\begin{proof}
Suppose there exists a non-identity automorphism $\varphi$ of $R$ with $g^\varphi\in \{g,g^{-1}\}$, for every $g\in R\setminus M$. If $R$ is generalised dihedral on $M$, then $g^\varphi=g$, for every $g\in R\setminus M$, because the elements in $R\setminus M$ are involutions. Therefore, $\varphi$ fixes $R\setminus M$ pointwise and hence $\varphi$ is the identity automorphism, which is a contradiction.

Conversely, suppose that $R$ is not generalised dihedral on $M$ and fix $a\in R\setminus M$.  Let $\psi:R\to R$ be the mapping  defined by $$m^\psi=m^{-1}\,\,\textrm{   and   }\,\,(ma)^\psi=m^{-1}a^{-1},\quad \forall m\in M.$$ We prove that $\psi$ is an automorphism, that is, $(g_1g_2)^\psi=g_1^\psi g_2^\psi$, for every $g_1,g_2\in R$. Let $m_1$ and $m_2$ be in $M$. Since $M$ is abelian, we have $(m_1m_2)^\psi=(m_1m_2)^{-1}=m_1^{-1}m_2^{-1}=m_1^{\psi}m_2^{\psi}$. Similarly, since $(m_1a)(m_2a)\in M$ and since $a^{-1}m_2^{-1}a^{-1}\in M$, we have $$((m_1a)(m_2a))^\psi=(m_1am_2a)^{-1}=a^{-1}m_2^{-1}a^{-1}m_1^{-1}=m_1^{-1}a^{-1}m_2^{-1}a^{-1}=g_1^\psi g_2^{\psi}.$$
Clearly, we have $$(m_1(m_2a))^\psi=(m_1m_2a)^\psi=(m_1m_2)^{-1}a^{-1}=m_1^{-1}(m_2^{-1}a^{-1})=m_1^\psi (m_2a)^{\psi}.$$
Finally, since $m_2^{a^{-1}}=am_2a^{-1}\in M$ and $a^{-2}\in M$, we have
\begin{eqnarray*}
((m_1a)m_2)^\psi&=&(m_1m_2^{a^{-1}}a)^\psi=m_1^{-1}(m_2^{a^{-1}})^{-1}a^{-1}=m_1^{-1}am_2^{-1}a^{-1}a^{-1}\\
&=&m_1^{-1}am_2^{-1}a^{-2}=m_1^{-1}aa^{-2}m_2^{-1}=m_1^{-1}a^{-1}m_2^{-1}=(m_1a)^\psi m_2^\psi.
\end{eqnarray*}
Therefore $\psi$ is indeed a group automorphism. Let $\iota_a:R\to R$ denote the inner automorphism of $R$ given by the conjugation via $a$ and let $\varphi:=\psi\iota_a$.
For every $m\in M$, we have $$(ma)^\varphi=(ma)^{\psi\iota_a}=(m^{-1}a^{-1})^{\iota_a}=a^{-1}m^{-1}=(ma)^{-1}$$ and hence $g^\varphi=g^{-1}$, for every $g\in R\setminus M$. The automorphism $\varphi$ is the identity automorphism if and only if $\psi=\iota_{a^{-1}}$,
that is, $m^{-1}=m^\psi=m^{a^{-1}}$,
for each $m\in M$, and $a^{-1}=a^\psi=a^{a^{-1}}=a$.
Thus $\varphi$ is the identity automorphism if and only if $R$ is the generalised dihedral group on $M$.
\end{proof}

\begin{lemma}\label{lemma:1}
Let $R$ be a finite group and let  $M$ be a subgroup of $R$ with $|R:M|=2$. Suppose that  $M$  contains an abelian subgroup $Z$ with $|M:Z|=2$ and there exists $a\in R\setminus M$, with $a^2\ne 1$, $a^2\in Z\cap \Zent R$ and $z^a=z^{-1}$, for every $z\in Z$. Then there exists a non-identity automorphism $\varphi$ of $R$ with $g^\varphi\in \{g,g^{-1}\}$, for every $g\in R\setminus M$.
\end{lemma}
\begin{proof}
Since $|R:M|=|M:Z|=2$ and $a\notin M$, we have $Z\unlhd M$ and $R=\langle M,a\rangle$. As $a$ normalises $Z$, we obtain $Z\unlhd R$. As $a^2\in Z$, we have $a^2=(a^2)^a=(a^2)^{-1}$ and hence $$a^4=1.$$
Observe also that $x^2=a^2$ and $z^x=z^{-1}$, for every $x\in Za$.

Fix $m\in M\setminus Z$. Since $a^2\in Z$ and $a\notin M$, we see that $Z$, $Za$, $Zm$, $Zam$ are the cosets of $Z$ in $R$. We define
$A:=\langle Z,a\rangle$ and $C:=\langle Z,am\rangle$. Observe that $A$, $C$ and $M$ are the three maximal subgroups of $R$ containing $Z$.

We define a mapping $\varphi:R\to R$ by
\[
g^\varphi=
\begin{cases}
g&\textrm{if }g\in C=Z\cup Zam,\\
g^{-1}&\textrm{if }g\in Za=A\setminus Z,\\
a^2g&\textrm{if }g\in Zm=M\setminus Z.\\
\end{cases}
\]
We prove that $\varphi$ is an automorphism of $R$, that is, for every $g_1,g_2\in R$, $(g_1g_2)^\varphi=g_1^\varphi g_2^{\varphi}$. For every $g\in R\setminus Zm$,  from the definition of $\varphi$, it is readily seen that $(g^\varphi)^{-1}=(g^{-1})^\varphi$. If $g\in Zm$, then $(g^{-1})^\varphi=a^2g^{-1}$ and $(g^\varphi)^{-1}=(a^2g)^{-1}=g^{-1}a^{-2}$. Since $a^4=1$ and $a^2\in \Zent R$, we deduce that $(g^{-1})^\varphi=(g^{-1})^\varphi$ also in this case. In particular, the equality
\begin{equation}\label{eq:0}
(g^{-1})^\varphi=(g^{-1})^\varphi
\end{equation} holds for every $g\in R$.

The restriction $\varphi_{|C}$ of $\varphi$ to $C$ is the identity mapping, which is an automorphism of $C$.

We show that the restriction $\varphi_{|A}$ of $\varphi$ to $A$ is also an automorphism of $A$. If $g_1,g_2\in Z$, then $(g_1g_2)^\varphi=g_1g_2=g_1^\varphi g_2^{\varphi}$. If $g_1\in Z,g_2\in Za$, then $g_1g_2\in Za$ and hence $$(g_1g_2)^\varphi=(g_1g_2)^{-1}=g_2^{-1}g_1^{-1}=(g_1^{-1})^{g_2}g_2^{-1}=g_1g_2^{-1}=g_1^\varphi g_2^\varphi.$$ If $g_1\in Za$ and $g_2\in Z$, then $g_1g_2\in Za$ and hence
$$(g_1g_2)^\varphi=(g_1g_2)^{-1}=g_2^{-1}g_1^{-1}=g_1^{-1}(g_2^{-1})^{g_1^{-1}}=g_1^{-1}g_2=g_1^\varphi g_2^\varphi.$$
Finally, if $g_1,g_2\in Za$, then $g_1g_2\in Z$ and hence $$(g_1g_2)^\varphi=g_1g_2=(g_1^{-1}g_1^2)(g_2^{-1}g_2^{2})=g_1^{-1}a^2g_2^{-1}a^2=g_1^{-1}g_2^{-1}a^4=g_1^{-1}g_2^{-1}=g_1^\varphi g_2^\varphi.$$

We show that the restriction $\varphi_{|M}$ of $\varphi$ to $M$ is an automorphism. If $g_1,g_2\in Z$, then $(g_1g_2)^\varphi=g_1g_2=g_1^\varphi g_2^{\varphi}$. If $g_1\in Zm$ and $g_2\in Z$, then $g_1g_2\in Zm$; since $a^4=1$ and $a^2$$\in \Zent R$,
we obtain $$(g_1g_2)^\varphi=a^2g_1g_2=(a^2g_1)g_2=g_1^\varphi g_2^\varphi.$$ If $g_1\in Z$ and $g_2\in Zm$, then $g_1g_2\in Zm$ and  $$(g_1g_2)^\varphi=a^2(g_1g_2)=g_1(a^2g_2)=g_1^\varphi g_2^\varphi.$$ Finally, if $g_1,g_2\in Zm$. Then $g_1g_2\in Z$ and hence $$(g_1g_2)^\varphi=g_1g_2=a^4g_1g_2=(a^2g_1)(a^2g_2)=g_1^\varphi g_2^\varphi.$$

For the rest of the proof we may suppose that $g_1,g_2$ are not both in $A$, or in $M$, or in $C$. Moreover, using~\eqref{eq:0}, we may reduce to consider only the following cases:
\begin{enumerate}
\item[Case 1:]$g_1\in Za$ and $g_2\in Zm$,
\item[Case 2:]$g_1\in Za$ and $g_2\in Zam$,
\item[Case 3:]$g_1\in Zm$ and $g_2\in Zam$.
\end{enumerate}

\smallskip

\noindent\textsc{Case 1: }$g_1=z_1a$ and $g_2=z_2m$, for some $z_1,z_2\in Z$.

\smallskip

\noindent Here, $g_1g_2\in Zam$ and hence $(g_1g_2)^\varphi=g_1g_2=z_1az_2m=z_1z_2^{a^{-1}}am=z_1z_2^{-1}am$, $g_1^\varphi=g_1^{-1}=a^{-1}z_1^{-1}$ and $g_2^\varphi=a^2g_2=a^2z_2m$. Therefore
$$g_1^\varphi g_2^\varphi=a^{-1}z_1^{-1}a^2z_2m=z_1z_2^{-1}a^{-1}a^2m=
z_1z_2^{-1}am=g_1g_2=(g_1g_2)^\varphi.$$

\smallskip

\noindent\textsc{Case 2: }$g_1=z_1a$ and $g_2=z_2am$, for some $z_1,z_2\in Z$.

\smallskip

\noindent Here, $g_1g_2\in Zm$ and hence
$$(g_1g_2)^\varphi=a^2g_1g_2=a^2z_1az_2am=a^{-1}z_1^{-1}z_2am=g_1^{-1}g_2=g_1^\varphi g_2^\varphi.$$

\smallskip

\noindent\textsc{Case 3: }$g_1=z_1m$ and $g_2=z_2am$, for some $z_1,z_2\in Z$.

\smallskip

\noindent Here, $g_1g_2\in Za$ and hence
$$(g_1g_2)^\varphi=g_2^{-1}g_1^{-1}=
m^{-1}a^{-1}z_2^{-1}m^{-1}z_1^{-1}.$$
Now, the element $a':=m^{-1}a^{-1}z_2^{-1}m^{-1}$ lies in $Za$ and hence it acts by conjugation on $Z$ inverting its elements. In particular, $a'z_1^{-1}=z_1a'$ and hence
$$(g_1g_2)^\varphi=z_1m^{-1}a^{-1}z_2^{-1}m^{-1}.$$
Since $a'\in Za$, we have $a'^2=a^2$ and hence $a'=a^2a'^{-1}=a^2mz_2am$.
In particular, we deduce
\begin{eqnarray*}
(g_1g_2)^\varphi=a^2z_1maz_2m=(a^2z_1m)(z_2am)=g_1^\varphi g_2^\varphi.
\end{eqnarray*}

\smallskip

Summing up, $\varphi$ is an automorphism of $R$. By construction, $Zam\subseteq \{g\in R\mid g^\varphi=g\}$ and $Za\subseteq \{g\in R\mid g^\varphi=g^{-1}\}$ and hence $g^\varphi\in \{g,g^{-1}\}$, for every $g\in R\setminus M$.

If $\varphi$ is the identity automorphism of $R$, then by the definition of $\varphi$ we infer $a^2g=g$, for every $g\in Zm$, that is, $a^2=1$. However, this contradicts the hypothesis that $a^2\ne 1$.
\end{proof}

\begin{lemma}\label{lemma:3}
Let $R$ be a finite group, let $M$ be a subgroup of $R$ with $|R:M|=2$, $|M:\Zent M|=4$ and $\gamma_2(M)=\langle a^2\rangle$ for some $a\in R\setminus M$  such that $o(a)=4$, $z^a=z^{-1}$ for every $z\in \Zent M$, and $o(am)\ne 2$ for some $m\in M\setminus \Zent M$.  Then there exists a non-identity automorphism $\varphi$ of $R$ with $g^\varphi\in \{g,g^{-1}\}$, for every $g\in R\setminus M$.
\end{lemma}
\begin{proof}
Let $Z:=\Zent M$. Write $M=Zm_0\cup Zm_1\cup Zm_2\cup Zm_3$, with $m_0:=1$ and for some $m_1,m_2,m_3\in M$. We define a mapping $\varphi:R\to R$ by
\[g^\varphi=
\begin{cases}
a^{-1}g^{-1}a^{-1}&\textrm{if }g\in M\setminus Z=Zm_1\cup Zm_2\cup Zm_3,\\
g&\textrm{if }g\in Z\cup Za,\\
g^{-1}&\textrm{if }g\in Zm_1a\cup Zm_2a\cup Zm_3a.
\end{cases}
\]
From the definition of $\varphi$, we have $Za\subseteq \cent R\varphi$ and $Zm_1a\cup Zm_2a\cup Zm_3a\subseteq \centm R \varphi$ and hence
$R\setminus M=Za\cup Zm_1a\cup Zm_2a\cup Zm_3a\subseteq \cent R\varphi\cup \centm R\varphi$. Therefore it remains to show that $\varphi$ is a non-identity automorphism of $R$.

Since $\langle a^2\rangle=\gamma_2(M)$ and $o(a^2)=2$, we have $a^2\in \Zent M$. Therefore, for every $g\in M$, we have $a^2g=ga^2$, that is, $aga=a^{-1}ga^{-1}$. From this it follows that, for every $g\in M\setminus Z$, we have $(g^\varphi)^{-1}=(a^{-1}g^{-1}a^{-1})=aga=a^{-1}ga^{-1}=(g^{-1})^\varphi$. In all other cases it is readily seen from the definition of $\varphi$ that this equality is also satisfied and hence
\begin{equation}\label{eq:-1}
(g^{-1})^\varphi=(g^\varphi)^{-1},\quad \forall g\in R.
\end{equation}

Define $A:=Z\cup Za$. Observe that $\varphi_{|A}$ is the identity mapping and hence it is an automorphism of $A$.

Next, we show that $\varphi_{|M}$ is an automorphism of $M$. If $g_1,g_2\in Z$, then $g_1g_2\in Z$ and hence $(g_1g_2)^\varphi=g_1g_2=g_1^\varphi g_2^\varphi$. If $g_1\in Z$ and $g_2\in M\setminus Z$, then $g_1g_2\in M\setminus Z$ and hence $(g_1g_2)^\varphi=a^{-1}(g_1g_2)^{-1}a^{-1}=a^{-1}g_2^{-1}g_1^{-1}a^{-1}$. Since $g_1\in Z=\Zent M$ and $g_1^{a}=g_1^{-1}$, we deduce
$$(g_1g_2)^\varphi=a^{-1}g_1^{-1}g_2^{-1}a^{-1}=g_1a^{-1}g_2^{-1}a^{-1}=g_1^\varphi g_2^\varphi.$$
If $g_1\in M\setminus Z$ and $g_2\in Z$, then the equality $(g_1g_2)^{\varphi}=g_1^\varphi g_2^\varphi$ holds by the previous case and by~\eqref{eq:-1}. Finally, suppose $g_1,g_2\in M\setminus Z$, that is,
$g_1=z_1m_i$ and $g_2=z_2m_j$, for some $z_1,z_2\in Z$ and for some $i,j\in \{1,2,3\}$. Let us
first consider the case that $i=j$. Then $g_1g_2\in Z$ and hence $(g_1g_2)^\varphi=g_1g_2$. On the
other hand, as $a^2\in \Zent M$, $o(a)=4$ and $(g_2g_1)^a=(g_2g_1)^{-1}$, we have
\begin{eqnarray*}
g_1^\varphi g_2^\varphi &=&
a^{-1}g_1^{-1}a^{-1}a^{-1}g_2^{-1}a^{-1}
=
a^{-1}g_1^{-1}a^{-2}g_2^{-1}a^{-1}=a^{-1}g_1^{-1}g_2^{-1}a^{-3}\\
&=&
a^{-1}(g_2g_1)^{-1}a=g_2g_1=g_1g_2=(g_1g_2)^\varphi.
\end{eqnarray*}
 Suppose
now that $i\ne j$ and let $k$$\in \{1,2,3\}$
with $\{1,2,3\}=\{i,j,k\}$. Now, $g_1g_2\in Zm_k$ and hence
\begin{eqnarray*}
(g_1g_2)^\varphi&=&
a^{-1}g_2^{-1}g_1^{-1}a^{-1}=a^{-1}m_j^{-1}z_2^{-1}m_i^{-1}z_1^{-1}a^{-1}\\
&=&a^{-1}m_j^{-1}m_i^{-1}z_1^{-1}z_2^{-1}a^{-1}=a^{-1}m_j^{-1}m_i^{-1}a^{-1}z_1z_2.
\end{eqnarray*}
Since the commutator subgroup of $M$ is $\langle a^2\rangle$, we obtain that $m_j^{-1}m_i^{-1}m_jm_i=a^2$ and hence
$m_j^{-1}m_i^{-1}=a^2m_i^{-1}m_j^{-1}$. Thus
$$(g_1g_2)^\varphi=am_i^{-1}m_j^{-1}a^{-1}z_1z_2.$$
On the other hand, with similar computations we obtain
\begin{eqnarray*}
g_1^\varphi g_2^\varphi&=a^{-1}m_i^{-1}z_1^{-1}a^{-2}m_j^{-1}z_2^{-1}a^{-1}=
am_i^{-1}m_j^{-1}z_1^{-1}z_2^{-1}a^{-1}=
am_i^{-1}m_j^{-1}a^{-1}z_1z_2=(g_1g_2)^\varphi.
\end{eqnarray*}
This shows that $\varphi_{|M}$ is an automorphism of $M$.

For the rest of the proof we may suppose that $g_1 , g_2$ are not both in $A$ or in $M$. Moreover,
using~\eqref{eq:-1}, we may reduce to consider only the following cases:
\begin{enumerate}
\item[Case 1:]$g_1\in Zm_1\cup Zm_2\cup Zm_3$ and $g_2\in Za$,
\item[Case 2:]$g_1\in Zm_1\cup Zm_2\cup Zm_3$ and $g_2\in Zam_1\cup Zam_2\cup Zam_3$,
\item[Case 3:]$g_1\in Zam_1\cup Zam_2\cup Zam_3$ and $g_2\in Zam_1\cup Zam_2\cup Zam_3$,
\item[Case 4:]$g_1\in Zam_1\cup Zam_2\cup Zam_3$ and $g_2\in Za$.
\end{enumerate}

\smallskip

\noindent\textsc{Case 1: }$g_1=z_1m_i$ and $g_2= z_2a$, for some $z_1,z_2\in Z$ and for some $i\in \{1,2,3\}$.

\smallskip

\noindent Here, $g_1g_2\in Zm_ia\subseteq Zam_1\cup Zam_2\cup Zam_3$ and hence $(g_1g_2)^{\varphi}=g_2^{-1}g_1^{-1}=a^{-1}z_2^{-1}m_i^{-1}z_1^{-1}$. Moreover, $g_1^\varphi=a^{-1} g_1^{-1}a^{-1}=a^{-1}m_i^{-1}z_1^{-1}a^{-1}$ and $g_2^\varphi=g_2=z_2a$. Thus
$$g_1^\varphi g_2^\varphi=a^{-1}m_i^{-1}z_1^{-1}a^{-1}z_2a=a^{-1}m_i^{-1}z_1^{-1}z_2^{-1}=a^{-1}z_2^{-1}m_i^{-1}z_1^{-1}=(g_1g_2)^\varphi.$$

\smallskip

\noindent\textsc{Case 2: }$g_1=z_1m_i$ and $g_2= z_2am_j$, for some $z_1,z_2\in Z$ and for some $i,j\in \{1,2,3\}$.

\smallskip

\noindent Clearly, $g_1g_2\in Zm_iam_j$. Suppose first $i$ and $j$ are such that $m_iam_j\in Za$. In particular, there exists $z\in Z$ with $m_iam_j=za$. Then $g_1g_2\in Za$ and hence $$(g_1g_2)^\varphi=g_1g_2=z_1m_i z_2am_j=z_1z_2m_iam_j=z_1z_2za.$$ Moreover,
$g_1^\varphi=a^{-1}g_1^{-1}a^{-1}=a^{-1}m_i^{-1}z_1^{-1}a^{-1}$ and $g_2^{\varphi}=g_2^{-1}=m_j^{-1}a^{-1}z_2^{-1}$. Thus
\begin{eqnarray*}
g_1^\varphi g_2^\varphi=a^{-1}m_i^{-1}z_1^{-1}a^{-1}m_j^{-1}a^{-1}z_2^{-1}.
\end{eqnarray*}
Since $m_iam_j=za$, we have $m_j^{-1}a^{-1}m_i^{-1}=a^{-1}z^{-1}$ and so $m_j^{-1}a^{-1}=a^{-1}z^{-1}m_i$ and hence
\begin{eqnarray*}
g_1^\varphi g_2^\varphi&=&
a^{-1}m_i^{-1}z_1^{-1}a^{-1}(a^{-1}z^{-1}m_i)z_2^{-1}=
a^{-1}m_i^{-1}z_1^{-1}a^{-2}z^{-1}m_iz_2^{-1}
=a^{-3}z_1^{-1}z^{-1}z_2^{-1}\\
&=&az_1^{-1}z_2^{-1}z^{-1}=z_1z_2za=(g_1g_2)^\varphi.
\end{eqnarray*}

Suppose next that $i$ and $j$ are such that $m_iam_j\in Zam_1\cup Zam_2\cup Zam_3$ and hence $m_iam_j=zam_k$, for some $z\in Z$ and for some $k\in \{1,2,3\}$. Then $g_1g_2\in Zam_1\cup Zam_2\cup Zam_3$ and hence $(g_1g_2)^\varphi=(g_1g_2)^{-1}=m_j^{-1}a^{-1}z_2^{-1}m_i^{-1}z_1^{-1}$. Moreover, $g_1^\varphi=a^{-1}g_1^{-1}a^{-1}=a^{-1}m_i^{-1}z_1^{-1}a^{-1}$ and $g_2^{\varphi}=g_2^{-1}=m_j^{-1}a^{-1}z_2^{-1}$.
Thus
$$(g_1g_2)^\varphi=m_j^{-1}a^{-1}z_2^{-1}m_i^{-1}z_1^{-1}=m_j^{-1}a^{-1}m_i^{-1}z_1^{-1}z_2^{-1}.$$
Now, as $a^{-1}m_j^{-1}a^{-1}\in M$, we obtain $z_1^{-1}(a^{-1}m_j^{-1}a^{-1})=(a^{-1}m_j^{-1}a^{-1})z_1^{-1}$ and hence
\begin{eqnarray*}
g_1^\varphi g_2^\varphi&=&a^{-1}m_i^{-1}z_1^{-1}a^{-1}m_j^{-1}a^{-1}z_2^{-1}=a^{-1}m_i^{-1}a^{-1}m_j^{-1}a^{-1}z_1^{-1}z_2^{-1}.
\end{eqnarray*}
Now, we apply the commutator formula $xy=yx[x,y]$ with $x:=m_i^{-1}$ and $y:=a^{-1}m_j^{-1}a^{-1}$ (for simplicity set $z':=[m_i^{-1},a^{-1}m_j^{-1}a^{-1}]$). We get
\begin{eqnarray*}
g_1^\varphi g_2^\varphi&=&a^{-1}(a^{-1}m_j^{-1}a^{-1})m_i^{-1}z'z_1^{-1}z_2^{-1}=a^{-2}m_j^{-1}a^{-1}m_i^{-1}z'z_1^{-1}z_2^{-1}\\
&=&m_j^{-1}a^{-1}m_i^{-1}z_1^{-1}z_2^{-1}z'a^{-2}=(g_1g_2)^\varphi z'a^{-2}.
\end{eqnarray*}
From this it  follows that $(g_1g_2)^\varphi=g_1^\varphi g_2^\varphi$ if and only if $z'=a^2$, that is, $[m_i^{-1},a^{-1}m_j^{-1}a^{-1}]=a^2$. Since $\gamma_2(M)=\langle a^2\rangle$, this happens if and only if $a^{-1}m_j^{-1}a^{-1}\notin Zm_i$, that is, $a^{-1}m_j^{-1}a^{-1}m_i^{-1}\notin Z$. Recall that $m_iam_j\in Zam_k$ and hence $a^{-1}m_j^{-1}a^{-1}m_i^{-1}\in a^{-1}(Zm_k)a\ne Z$, because  $k\in \{1,2,3\}$.

\smallskip

\noindent\textsc{Case 3: }$g_1=z_1am_i$ and $g_2= z_2am_j$, for some $z_1,z_2\in Z$ and for some $i,j\in \{1,2,3\}$.

\smallskip

\noindent Here, $g_1g_2\in Zam_iam_j\subseteq M=Z\cup Zm_1\cup Zm_2\cup Zm_3$. We distinguish two cases depending on whether $am_iam_j\in Z$ or $am_iam_j\in Zm_1\cup Zm_2\cup Zm_3$. We start with the first case: $am_iam_j=z$, for some $z\in Z$. Thus $g_1g_2\in Z$ and hence $$(g_1g_2)^\varphi=g_1g_2=z_1am_iz_2am_j=z_1am_iam_jz_2^{-1}=z_1zz_2^{-1}.$$
On the other hand, $g_1^{\varphi}=g_1^{-1}$ and $g_2^\varphi=g_2^{-1}$ and hence
\begin{eqnarray*}
g_1^\varphi g_2^\varphi&=&m_i^{-1}a^{-1}z_1^{-1}m_j^{-1}a^{-1}z_2^{-1}.
\end{eqnarray*}
Since $am_iam_j=z$, we get $m_j^{-1}a^{-1}=z^{-1}am_i$ and hence
\begin{eqnarray*}
g_1^\varphi g_2^\varphi&=&m_i^{-1}a^{-1}z_1^{-1}(m_j^{-1}a^{-1})z_2^{-1}=
m_i^{-1}a^{-1}z_1^{-1}(z^{-1}am_i)z_2^{-1}\\
&=&m_i^{-1}(z_1^{-1}z^{-1})^am_iz_2^{-1}
=
m_i^{-1}z_1zm_iz_2^{-1}=z_1zz_2^{-1}=(g_1g_2)^\varphi.
\end{eqnarray*}

Finally, suppose that $am_iam_j=zm_k$, for some $k\in \{1,2,3\}$. In this case, we have $g_1g_2\in Zm_k$ and hence $(g_1g_2)^\varphi=a^{-1}(g_1g_2)^{-1}a^{-1}$ and hence
\begin{eqnarray*}
(g_1g_2)^\varphi&=&a^{-1}m_j^{-1}a^{-1}z_2^{-1}m_i^{-1}a^{-1}z_1^{-1}a^{-1}=a^{-1}m_j^{-1}a^{-1}z_2^{-1}m_i^{-1}(z_1^{-1})^aa^{-2}\\
&=&a^{-1}m_j^{-1}a^{-1}m_i^{-1}z_1z_2^{-1}a^{-2}.
\end{eqnarray*}
On the other hand,
$g_1^{\varphi}=g_1^{-1}$ and $g_2^\varphi=g_2^{-1}$ and hence
\begin{eqnarray*}
g_1^\varphi g_2^\varphi&=&m_i^{-1}a^{-1}z_1^{-1}m_j^{-1}a^{-1}z_2^{-1}=
m_i^{-1}a^{-1}m_j^{-1}a^{-1}z_1z_2^{-1}.
\end{eqnarray*}
Now, we apply the commutator formula $xy=yx[x,y]$ with $x:=m_i^{-1}$ and $y:=a^{-1}m_j^{-1}a$; for simplicity write $z':=[m_i^{-1},a^{-1}m_j^{-1}a^{-1}]$. We obtain
\begin{eqnarray*}
g_1^\varphi g_2^\varphi&=&m_i^{-1}(a^{-1}m_j^{-1}a^{-1})z_1z_2^{-1}=(a^{-1}m_j^{-1}a^{-1})m_i^{-1}z'z_1z_2^{-1}\\
&=&a^{-1}m_j^{-1}a^{-1}m_i^{-1}z_1z_2^{-1}z'=(g_1g_2)^\varphi a^{-2}z'.
\end{eqnarray*}
From this it  follows that $(g_1g_2)^\varphi=g_1^\varphi g_2^\varphi$ if and only if $z'=a^2$, that is, $[m_i^{-1},a^{-1}m_j^{-1}a^{-1}]=a^2$. Since $\gamma_2(M)=\langle a^2\rangle$, this happens if and only if $a^{-1}m_j^{-1}a^{-1}\notin Zm_i$, that is, $a^{-1}m_j^{-1}a^{-1}m_i^{-1}\notin Z$. Recall that $am_iam_j\in Zm_k$ and hence $m_j^{-1}a^{-1}m_i^{-1}a^{-1}\in Zm_k$; therefore $a^{-1}m_j^{-1}a^{-1}m_i^{-1}=(m_j^{-1}a^{-1}m_i^{-1}a^{-1})^a\notin Z$,  because  $k\in \{1,2,3\}$.

\smallskip

\noindent\textsc{Case 4: }$g_1=z_1am_i$ and $g_2= z_2a$, for some $z_1,z_2\in Z$ and for some $i\in \{1,2,3\}$.

\smallskip

\noindent Here, $g_1g_2\in Z\cup Zm_1\cup Zm_2\cup Zm_3=M$. Set $g_1':=g_1$ and $g_2':=g_2m_1$. We have $(g_1'g_2')^\varphi=g_1'^\varphi g_2'^\varphi$, from Case~3 applied to $g_1'$ and $g_2'$. Moreover, $g_2'^\varphi=(g_2m_1)^\varphi=g_2^\varphi m_1^\varphi$, from Case~1 and~\eqref{eq:-1}. Since $g_1g_2=(g_1'g_2')m_1^{-1}$ and $g_1'g_2'\in Z\cup Zm_1\cup Zm_2\cup Zm_3$, we deduce from the fact that $\varphi_{|M}$ is an automorphism that
\begin{align*}
(g_1g_2)^\varphi&=(g_1'g_2')^\varphi (m_1^{-1})^\varphi=g_1'^\varphi g_2'^\varphi(m_1^{-1})^\varphi=g_1^\varphi g_2^\varphi m_1^\varphi(m_1^{-1})^\varphi=g_1^\varphi g_2^\varphi.
\end{align*}

Summing up we have shown that $\varphi$ is an automorphism of $G$ and it remains to show that $\varphi$ is not the identity. If $\varphi$ is the identity automorphism, then $g^{-1}=g$, for every $g\in Zam_1\cup Zam_2\cup Zam_3$. However, this contradicts the fact that $o(am)\ne 2$, for some $a\in M\setminus \Zent M$.
\end{proof}

\begin{proof}[Proof of Theorem~$\ref{thrm:11}$]
From Lemmas~\ref{example:1},~\ref{lemma:1} and~\ref{lemma:3}, we see that, if part~\eqref{thrmeq:1},~\eqref{thrmeq:2} or~\eqref{thrmeq:3} holds, then $R$ admits a non-identity automorphism $\varphi$ of $R$ with $g^\varphi\in \{g,g^{-1}\}$, for every $g\in R\setminus M$. We now need to prove the converse and hence we suppose there exists a non-identity automorphism $\varphi$ of $R$ with $g^\varphi\in \{g,g^{-1}\}$, for every $g\in R\setminus M$.
In particular, $R=M\cup \cent R \varphi\cup\centm R\varphi$. Assume first that $\cent R\varphi\le M$. Then $R=M\cup \centm R\varphi$ and hence
$$|\centm R\varphi|\ge |R\setminus M|=|R|/2.$$
In particular, $R$ is a group admitting an automorphism inverting at least half of its elements. In~\cite{F,HeMa,LM,Potter}, these groups are refereed to as $1/2$-groups. Strictly speaking, we do not need the classification of the $1/2$-groups arising from the work in~\cite{F,HeMa,LM,Potter}, however our elementary argument owns a great deal to some of the arguments therein.
Let $a\in R\setminus M$ and let $m\in M$. Since $\varphi$ inverts each element in $R\setminus M$, we deduce
$$m^\varphi=(maa^{-1})^\varphi=(ma)^\varphi(a^{-1})^\varphi=a^{-1}m^{-1}a.$$
In particular, the restriction $\varphi_{|M}$ of $\varphi$ to $M$ is given by the mapping defined by $m\mapsto a^{-1}m^{-1}a$, for every $m\in M$. Since $\varphi_{|M}$ does not depend upon $a\in R\setminus M$, we deduce that$$a_1^{-1}m^{-1}a_1=m^\varphi=a_2^{-1}m^{-1}a_2,\quad \forall a_1,a_2\in R\setminus M,\forall m\in M.$$  Therefore, $a_2a_1^{-1}$ centralises $M$. Since $a_1$ and $a_2$ are two arbitrary elements of $R\setminus M$ and since $|R:M|=2$, we deduce that $a_2a_1^{-1}$ is an arbitrary element of $M$ and hence $M$ is abelian. Therefore, Lemma~\ref{example:1} shows that $R$ is not generalised abelian over $M$ and we obtain part~\eqref{thrmeq:1}.

\smallskip

Assume then $\cent R \varphi\nleq M$. As $\varphi$ is not the identity automorphism, we also have $M\nleq \cent R \varphi$.
 For simplicity, we set $C:=\cent R \varphi$ and $C^-:=\centm R\varphi$ and we define $$Z:=M\cap C.$$
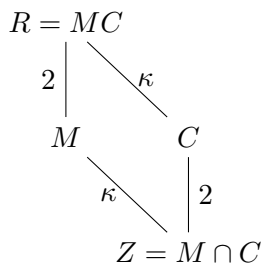
\begin{figure}[!h]
\begin{tikzpicture}[node distance   = 1cm ]
        \node(A){$R=MC$};
\node[below=of A](B){$M$};
\node[right=of B](C){$C$};
\node[below=of C](D){$Z=M\cap C$};
\draw(A)--node[left]{$2$}(B);
\draw(A)--node[right]{$\kappa$}(C);
\draw(D)--node[left]{$\kappa$}(B);
\draw(D)-- node[right]{$2$}(C);
\end{tikzpicture}
\caption{Subgroup lattice for $R$}\label{fig:1}
\end{figure}
In particular, $Z$ is a proper subgroup of $C$ and of $M$. Since $C\nleq M$, there exists $a\in C\setminus Z$ and hence $R=M\langle a\rangle$. Let $m\in M\setminus Z$. Then $ma\in R\setminus M\subseteq C\cup C^-$. Since $a\in C$ and $m\notin C$, we deduce $ma^{-1}\in C^-$ and hence
$(ma^{-1})^\varphi=am^{-1}$. On the other hand, since $\varphi$ is an automorphism, we have $(ma^{-1})^\varphi=m^\varphi (a^{-1})^\varphi=m^\varphi a^{-1}$. From this, we obtain
\begin{equation}\label{eq:1}
m^\varphi=am^{-1}a,\qquad\forall m\in M\setminus Z, \forall a\in A\setminus Z.
\end{equation}

Let $m\in M\setminus Z$ and let $z\in Z$. Applying~\eqref{eq:1} with $m$ replaced by $mz$, we obtain $(mz)^\varphi=az^{-1}m^{-1}a$. On the other hand, since $\varphi$ is an automorphism, $(mz)^\varphi=m^\varphi z^\varphi=am^{-1}az$. It follows that
\begin{equation}\label{eq:2}
(am)z(am)^{-1}=z^{-1},\qquad\forall z\in Z, \forall m\in M\setminus Z.
\end{equation}
From~\eqref{eq:2}, $Z$ is abelian  because the automorphism given by the conjugation via $am$ inverts each of its elements. Now, fix for the rest of the proof $a\in A\setminus Z$.

Let $\kappa:=|M:Z|$ and let $m_0,m_1,\ldots,m_{\kappa-1}$ be a set of representatives for the cosets of $Z$ in $M$ with $m_0:=1$. For every $i,j\in \{1,\ldots,\kappa-1\}$ with $i\ne j$ and for every $z_i,z_j\in Z$, we have $m_iz_i,m_jz_j,m_iz_i(m_jz_j)^{-1}\in M\setminus Z$ and hence, from~\eqref{eq:1}, we obtain $$(m_iz_iz_j^{-1}m_j^{-1})^{\varphi}=am_jz_jz_i^{-1}m_i^{-1}a.$$ Since $\varphi$ is an automorphism, we have $$(m_iz_i(m_jz_j)^{-1})^\varphi=(m_iz_i)^\varphi ((m_jz_j)^{-1})^\varphi=az_i^{-1}m_i^{-1}a^{2}m_jz_ja.$$ Therefore
$m_jz_jz_i^{-1}m_i^{-1}=z_i^{-1}m_i^{-1}a^{2}m_jz_j$. Rearranging the terms of this equality, we obtain
\begin{equation}\label{eq:3}
a^2=[(m_iz_i)^{-1},(m_jz_j)^{-1}],\qquad \forall i,j\in \{1,\ldots,\kappa-1\} \textrm{ with }i\ne j, \forall z_i,z_j\in Z.
\end{equation}
Observe that the right hand side of~\eqref{eq:3} does not depend on $z_i$ and $z_j$, or on $i,j\in \{1,\ldots,\kappa-1\}$ with $i\ne j$.

\smallskip

Suppose  $\kappa=2$. From~\eqref{eq:2} we have $z^{am_1}=z^{-1}$, for every $z\in Z$. Moreover, since $M/Z$ and $C/Z$ are two distinct subgroups of $R/Z$, we deduce that $R/Z$ is not cyclic and hence $(am_1)^2\in Z$.

If $(am_1)^2=1$, then, for every $z\in Z$, by~\eqref{eq:2} we deduce $$(am_1z)^{\varphi}=(am_1z)^{-1}=z^{-1}m_1^{-1}a^{-1}=z^{-1}(am_1)^{-1}=z^{-1}(am_1)=(am_1)z=am_1z,$$ that is, $\varphi$ fixes $Zam_1$ pointwise. Thus $\varphi$ centralises $Zam_1$ and $Za$ and hence also $\langle Z,am_1,a\rangle=R$, contradicting the fact that $\varphi$ is not the identity automorphism. Therefore $(am_1)^2\ne 1$.

From~\eqref{eq:2} and from $(am_1)^2\ne 1$, it follows that all elements in $Zam_1$ square to $(am_1)^2$. As $|R:\langle Z,am_1\rangle|=2$, we deduce $\langle Z,am_1\rangle\unlhd R$. In particular, for every $g\in R$, $(am_1)^g\in Zam_1$ and hence $(am_1)^g$ squares to $(am_1)^2$, that is, $((am_1)^\varphi)^2=((am_1)^2)^g=(am_1)^2$. Therefore, $(am_1)^2\in \Zent R$.

We have shown that part~\eqref{thrmeq:2} holds (with the element $a$ in the statement of Theorem~\ref{thrm:1} part~\eqref{thrmeq:1} replaced by $am_1$ here).  In particular, for the rest of the proof we may suppose that $$\kappa\ge 3.$$

\smallskip

Applying~\eqref{eq:3} with $z_j=1$ and using the commutator formula $[xy,z]=[x,z]^y[y,z]$, we obtain
$$a^2=[z_i^{-1}m_i^{-1},m_j^{-1}]=[z_i^{-1},m_j^{-1}]^{m_i^{-1}}[m_i^{-1},m_j^{-1}]=[z_i^{-1},m_j^{-1}]^{m_i^{-1}}a^2,$$
and hence $[z_i^{-1},m_j^{-1}]=1$. Since $i$ is an arbitrary index in $\{1,\ldots,\kappa-1\}$, we deduce that $z_i\in \Zent M$ and since $z_i$ is an arbitrary element in $Z$, we deduce that $$Z\le \Zent M.$$

Recall the commutator formula $[y,x]=[x,y]^{-1}$. From~\eqref{eq:3} we deduce $$a^2=[m_2^{-1},m_1^{-1}]=[m_1^{-1},m_2^{-1}]^{-1}=(a^{2})^{-1}$$
and hence $o(a)\in \{2,4\}$.

As $Z\le \Zent M$, if $a^2=1$, then~\eqref{eq:3}  yields that $M$ is abelian. In particular, part~\eqref{thrmeq:1} holds from Lemma~\ref{example:1}. Therefore, we may suppose $$o(a)=4.$$ Now,~\eqref{eq:3} yields $$Z=\Zent M\textrm{ and }\gamma_2(M)=\langle a^2\rangle.$$

Assume there exists $i\in \{1,\ldots,\kappa-1\}$ with $m_i^2\notin Z$. Then $m_i^2z=m_j$, for some $j\in \{1,\ldots,\kappa-1\}$ and for some $z\in Z$. As $Z= \Zent M$,~\eqref{eq:3} yields $a^2=[m_i^{-1},m_j^{-1}]=[m_i^{-1},z^{-1}m_i^{-2}]=[m_i^{-1},m_i^{-2}]=1$, contradicting the fact that $a^2\ne 1$. Therefore $m_i^2\in Z$, for every $i\in \{1,\ldots,\kappa-1\}$ and hence $M/Z$ is an elementary abelian $2$-group and $\kappa$ is a power of $2$. In particular, $\kappa\ge 4$.

Suppose $\kappa> 4$. Choose $i,j,k\in \{1,\ldots,\kappa-1\}$ such that $m_i,m_j,m_im_j^{-1}$ are in distinct $Z$-cosets: observe that this is possible because $\kappa> 4$. Now~\eqref{eq:3} yields
\begin{eqnarray*}
a^2=[(m_im_j^{-1})^{-1},m_k^{-1}]=[m_jm_i^{-1},m_k^{-1}]=[m_j,m_k^{-1}]^{m_i^{-1}}[m_i^{-1},m_k^{-1}]=(a^2)^{m_i^{-1}}a^2
\end{eqnarray*}
and hence $a^2=1$, which is a contradiction. Therefore, $\kappa=4$. If $o(am)=2$ for every $m\in M$, then~\eqref{eq:1} yields $M\le \cent G\varphi= C$, which is a contradiction. Therefore, part~\eqref{thrmeq:3} holds.
\end{proof}

The core of our argument for the proof of Theorem~\ref{thrm:1} is the work in~\cite{DSV} on the automorphism group of Cayley graphs over abelian groups. We find that this is a useful paper for this type of investigations and recently it was also used for investigating the distinguishing number of certain Cayley graphs, see~\cite{BSS}.


\thebibliography{10}
\bibitem{babai1}L.~Babai, Finite digraphs with given regular automorphism groups, \textit{Periodica Mathematica Hungarica} \textbf{11} (1980), 257--270.


\bibitem{babai2}L.~Babai, W.~Imrich, Tournaments with given regular group, \textit{Aequationes Mathematicae} \textbf{19} (1979), 232--244.

\bibitem{BSS}N.~Balachandran, S.~Padinhatteeri, P.~Spiga, Vertex transitive graphs $G$ with $\chi_D(G)>\chi(G)$ and small automorphism group, \textit{Ars Math. Contemporanea}, to appear.
\bibitem{DSV}E.~Dobson, P.~Spiga, G.~Verret, Cayley graphs on abelian groups, \textit{Combinatorica} \textbf{36} (2016), 371--393.

\bibitem{WT1}J.~K.~Doyle, T.~W.~Tucker, and M.~E.~Watkins, Graphical Frobenius Representations, \textit{J. Algebraic Combin. }\textbf{48} (2018), 405--428.

\bibitem{DFS}J.-L.~Du, Y.~-Q.~Feng, P.~Spiga, On the existence and the enumeration of bipartite regular representations of Cayley graphs over abelian groups, \textit{submitted}.

\bibitem{F}P.~Fitzpatrick, Groups in which an automorphism inverts precisely half of the elements, \textit{Proc. Roy. Irish. Acad. Sect. A} \textbf{86} (1986), 81--89.

\bibitem{Godsil}C.~D.~Godsil, GRRs for nonsolvable groups, Algebraic methods in graph theory, Vol. I, II
(Szeged, 1978), pp. 221–239, Colloq. Math. Soc. J\'{a}nos Bolyai, Amsterdam-New York, 1981.

\bibitem{HeMa}P.~Hegarty, D.~MacHale, Two-groups in which an automorphism inverts precisely half of the elements, \textit{Bull. London Math. Soc. }\textbf{30} (1998), 129--135.


\bibitem{Hetzel}D.~Hetzel, \"{U}ber regul\"{a}re graphische Darstellung von aufl\"{o}sbaren Gruppen, Technische Universit\"{a}t, Berlin, 1976.

\bibitem{IW}W. Imrich, M. Watkins, On graphical regular representations of cyclic extensions of groups, \textit{Pacific J. Math. }\textbf{54} (1974), 1--17.


\bibitem{LM}H.~Liebeck, D.~MacHale, Groups with automorphisms inverting most elements, \textit{Math. Z.} \textbf{124} (1972), 51--63.


\bibitem{morrisspiga}J.~Morris, P.~Spiga, Every Finite Non-Solvable Group admits an Oriented Regular Representation, \textit{Journal of Combinatorial Theory Series B} \textbf{126} (2017), 198--234.

\bibitem{morrisspiga1}J.~Morris, P.~Spiga, Classification of finite groups that admit an oriented regular representation, \textit{Bull. London Math. Soc.} \textbf{50} (2018), 1--21.


\bibitem{nowitz1}L.~A.~Nowitz, Graphical regular representations of non-abelian groups. I. \textit{Canadian J. Math.} \textbf{24} (1972), 993--1008.

\bibitem{nowitz2}L.~A.~Nowitz, Graphical regular representations of non-abelian groups. II. \textit{Canadian J. Math.} \textbf{24} (1972),  1009--1018.
\bibitem{nowitz3}L.~A.~Nowitz, M.~E.~Watkins,  On graphical regular representations of direct products of groups, \textit{Monatsh. Math.} \textbf{76} (1972), 168--171.



\bibitem{LPSLPS}M.~W.~Liebeck, L.~Pyber, A.~Shalev, On a conjecture of G. E. Wall, \textit{J. Algebra} \textbf{317} (2007), 184--197.
\bibitem{Potter}W.~M.~Potter, Nonsolvable groups with an automorphism inverting many elements, \textit{Archiv der Mathematik} \textbf{50} (1988),  292--299.


\bibitem{spiga}P. Spiga, Finite groups admitting an oriented regular representation, \textit{J. Comb. Theory Series A} \textbf{153} (2018), 76--97.


\bibitem{spiga0}P.~Spiga, On the existence of Frobenius digraphical representations, \textit{Electron. J. Combin. }\textbf{25} (2018), no. 2, Paper~2.6, 19~pp.

\bibitem{spiga1}P.~Spiga, On the existence of graphical Frobenius representations and their asymptotic enumeration: an answer to the GFR conjecture, \textit{submitted}.


\bibitem{Wall}G.~E.~Wall, Some applications of the Eulerian functions of a finite group, \textit{J. Aust. Math. Soc.} \textbf{2} (1961), 35--59.


\bibitem{nowitz4}M.~E.~Watkins, On the action of non-Abelian groups on graphs, \textit{J. Combinatorial Theory Ser. B} \textbf{11} (1971), 95--104.

\bibitem{W2}M. E. Watkins, On graphical regular representations of $C_n \times Q$, in: Y. Alavi, D. R. Lick and A. T. White (eds.), Graph Theory and Applications, Springer, Berlin, volume 303 of Lecture Notes in Mathematics, pp.305--311, 1972, proceedings of the Conference at Western Michigan University, Kalamazoo, Michigan, May 10-13, 1972 (dedicated to the memory of J. W. T. Youngs).


\bibitem{W3} M. E. Watkins, Graphical regular representations of alternating, symmetric, and miscellaneous small groups, \textit{Aequationes Math.} \textbf{11} (1974), 40--50.
\end{document}